\documentclass[11pt]{article}
\usepackage{amssymb,amsmath,accents}
\usepackage{amscd}
\usepackage{amsfonts,amsthm,mathrsfs}
\usepackage{setspace}
\usepackage{cases,empheq}
\usepackage{subcaption}
\usepackage{diagbox}
\usepackage{url}

\usepackage{graphicx, xcolor}

\numberwithin{equation}{section}
\newtheorem{thm}{Theorem}
\newtheorem{cor}[thm]{Corollary}
\newtheorem{lem}[thm]{Lemma}
\newtheorem{prop}[thm]{Proposition}

\newtheorem{definition}[thm]{Definition}
\newtheorem{rem}[thm]{Remark}

\begin{document}

\title{Remarks on temperature-dependent Cahn--Hilliard equations \vspace{0.2em}\\
\normalsize{Dedicated to Professor Chun Liu on the occasion of his 60th birthday}}

\author{Mi-Ho Giga$^1$, Yoshikazu Giga$^2$, Naoto Kajiwara$^3$ and Kazuyuki Tsuda$^4$ \vspace{0.7em}\\ 
\centerline{{\footnotesize $^1$Graduate School of Mathematical Sciences}}\vspace{-0.3em}\\
\centerline{{\footnotesize The University of Tokyo}}\vspace{-0.3em}\\
\centerline{{\footnotesize 3-8-1 Komaba, Meguro-ku, Tokyo 153-8914, Japan}}\vspace{-0.3em}\\
\centerline{{\footnotesize mihogiga-ec@g.ecc.u-tokyo.ac.jp}} \vspace{0.7em}\\
\centerline{{\footnotesize $^2$Graduate School of Mathematical Sciences}}\vspace{-0.3em}\\
\centerline{{\footnotesize The University of Tokyo}}\vspace{-0.3em}\\
\centerline{{\footnotesize 3-8-1 Komaba, Meguro-ku, Tokyo 153-8914, Japan}}\vspace{-0.3em}\\
\centerline{{\footnotesize labgiga@ms.u-tokyo.ac.jp}} \vspace{0.7em}\\
\centerline{{\footnotesize $^3$Applied Physics Course, Department of Electrical, Electronic and Computer Engineering}}\vspace{-0.3em}\\
\centerline{{\footnotesize Gifu University}}\vspace{-0.3em}\\
\centerline{{\footnotesize 1-1 Yanagido, Gifu City, Gifu 501-1193, Japan}} \vspace{-0.3em}\\
\centerline{{\footnotesize kajiwara.naoto.p4@f.gifu-u.ac.jp}} \vspace{0.7em}\\
\centerline{{\footnotesize $^4$Center for Fundamental Education}}\vspace{-0.3em}\\
\centerline{{\footnotesize Kyushu Sangyo University}}\vspace{-0.3em}\\
{\scriptsize 2-3-1 Matsukadai, Higashi-ku, Fukuoka 813-8503, Japan}\vspace{-0.3em}\\
\centerline{{\footnotesize k-tsuda@ip.kyusan-u.ac.jp}}}
\date{}
\maketitle

\renewcommand{\thefootnote}{\fnsymbol{footnote}}
\footnote[0]{MSC 2020 35Q79; 74N20, 80A17, 80A22} 
\footnote[0]{keywords temperature-dependent Cahn-Hilliard equation, the second law of thermodynamics, unique solvability, stability of equilibria.}
\renewcommand{\thefootnote}{\arabic{footnote}}

\begin{abstract}
We recall several derivations of temperature-dependent Cahn-Hilliard equations.
 The resulting equations are systems of mixed order.
 For a suitable subclass of equations we discuss its local-in-time well-posedness based on $L^p$-maximal regularity theory.
 We also discuss stability of equilibria.
\end{abstract}
\maketitle

\section{Introduction} \label{SI}

The Cahn-Hilliard equation is a typical model to describe phase separation.
 However, it is often discussed in isothermal case; see for example a book by A.~Miranville \cite{Mi}.
 A recent seminal paper by De Anna et al.\ \cite{DLSS} derived two temperature-dependent Cahn-Hilliard equations which fulfill the second law of thermodynamics.

The purpose of this paper is threefold.
 We derive temperature-dependent Cahn-Hilliard equations in a way different from \cite{DLSS} but more classical.
 Then we discuss local-in-time solvability of these equations.
 We finally discuss stability of an equilibrium which may not be constant.

Let us first recall an isothermal Cahn-Hilliard equation.
 Let $\Omega$ be a domain in $\mathbb{R}^n$ where the binary mixture of materials (fluids) occupies.
 It consists of two substances and (local) mass fraction (or concentration) of the first substance is denoted by $w=w(x,t)$ which depends on $x\in\Omega$ and time $t\ge0$.
 This unknown $w$ is often called the order parameter or the phase field variable.
 A dimensionless form of the classical Cahn-Hilliard equation is of the form
\[
    \partial_t w=\Delta\mu, \quad
    \mu=-\Delta w+\frac{\partial\psi_0}{\partial w}(w).
\]
Here $\psi_0=\psi_0(w)$ denotes the bulk energy density.
 The quantity $\mu$ is a (generalized) chemical potential.
 If we introduce
\[
    \psi(w,\nabla w)
    =\psi_0(w)+\frac12|\nabla w|^2,
\]
the chemical potential $\mu$ can be written as
\begin{align*}
    \mu=\delta_w\psi:=&
    \psi_w-\operatorname{div}\psi_{\nabla w}\\
    =&\psi_{0w}-\Delta w
\end{align*}
where $\psi_w=\partial_w\psi=\partial\psi/\partial w$ and
\[
    \psi_{\nabla w}=\left.\nabla_z \psi(w,z)\right|_{z=\nabla w}, \quad
    \nabla_z\psi(w,z)=(\psi_{z_1},\ldots,\psi_{z_n}).
\]
The system was originally introduced in \cite{C}, \cite{CH1}, \cite{CH2} to describe a process of phase separation by considering a double-well potential $\psi_0$.

\vspace{0.5em}\noindent
\textbf{Derivation.}
 To include temperature we take a classical approach.
 We start with macroscopic motionless thermal fluid.
 Let $\psi=\psi(w,\theta,\nabla w)$ be the Helmholtz energy and $\eta=\eta(w,\theta,\nabla w)$ be the corresponding entropy defined by $\eta=-\psi_\theta$, where $\theta$ denotes the temperature which is a positive quantity.
 Let $\epsilon=\psi+\theta\eta$ denote the internal energy.
 We start with
\begin{equation} \label{EIC}
    \partial_t w=-\operatorname{div}j, \quad
    \partial_t\epsilon=\operatorname{div}j_e,
\end{equation}
where $j$ denotes the mass flux and $j_e$ denotes the energy flux.
 (As in Section~\ref{SD}, this system itself is derived from conservation of mass, momentum and energy.)
 We consider the heat flux $q$.
 The second law of thermodynamics leads
\begin{equation} \label{EIE}
    \partial_t \eta+\operatorname{div}j_s=\Delta\ge0,
\end{equation}
where $j_s$ is the \emph{entropy flux}.
 There are several choices of $j_e$ and $j_s$.
 We discuss three types
\begin{description}
    \item[\rm Type I] $j_s=q/\theta$ while $j_e=-q+j_m$ with some flux $j_m$;
    \item[\rm Type I\hspace{-0.1em}I] $j_e=-q$ while $j_s=q/\theta+j_\alpha$ with some flux $j_\alpha$;
    \item[\rm Type I\hspace{-0.1em}I\hspace{-0.1em}I] $j=\bar{u}w$, $j_s=\bar{u}\eta+q/\theta$ with some vector field $\bar{u}$, where $j_e=-q+j_m$ for some flux $j_m$.   
\end{description}
In each case, we take $j_m$, $j_\alpha$, $\bar{u}$ and $j$ in a suitable way to make $\Delta\ge0$.
 We also use ``maximizing entropy principle (\cite{HM}, \cite{MP})" to determine $j$.
 We often assume Fourier's law $q=-d(w,\theta)\nabla\theta$ for the heat flux which also includes the case $q=\nabla(1/\theta)$ by taking $d=\theta^{-2}$.

In \cite{DLSS}, Type I and Type I\hspace{-0.1em}I\hspace{-0.1em}I are derived by assuming Fourier's law with constant $d$.
 Instead of starting from \eqref{EIC}, they started from
\[
    \partial_t w+\operatorname{div}(\bar{u}w)=0,
\]
where $\bar{u}$ is microscopic velocity.
 In \cite{DLSS}, they introduced total energy and dissipation energy to derive conservative force and dissipative force.
 The microscopic velocity $\bar{u}$ is obtained by balance of both forces.
 This idea is used in many problems especially isothermal case; see e.g.\ a review paper \cite{GKL}.
 In other words, $j$ is determined first.
 Once $j$ is fixed, then \eqref{EIE} yields the desired equations.
 For Type I, the resulting equation is
\begin{equation}
    \begin{gathered} \label{EICHE}
    \partial_t w=-\operatorname{div}j, \quad
    j=-\gamma\nabla\mu,\quad
    \mu=\delta_w \psi \\
    \partial_t\epsilon=\operatorname{div}(d\nabla\theta)-\operatorname{div}(\mu j+\psi_{\nabla w}\operatorname{div}j)
    \end{gathered}
\end{equation}
with
\[
    \partial_t\eta+\operatorname{div}(q/\theta)=\Delta, \quad
    \Delta=d|\nabla\theta|^2/\theta^2+\gamma|\nabla\mu|^2/\theta,
\]
where $d\ge0$, $\gamma\ge0$.
 In \cite{DLSS}, $\nabla\mu$ is replaced by $\nabla\mu+a\nabla\partial_tw$ with $a\ge0$ since they introduce a regularizing term in their dissipation energy.
 Then resulting equation is regularized by this modification.
 In the case of Type I,
\[
    \theta\Delta=d|\nabla\theta|^2/\theta
    +\gamma|\nabla\mu+a\nabla\partial_tw|^2
    +a|\partial_tw|^2
\]
in \cite{DLSS}.
 The assumption Type I\hspace{-0.1em}I is also popular as derived by \cite{AP} and \cite{FK1}. 
 In the case that $\psi$ is independent of $w$, Type I\hspace{-0.1em}I is also discussed in \cite{BD} and \cite{BS}.

Let us explain our derivation which is along the line of \cite{MP}.
 We shall start with \eqref{EIC} and calculate the entropy production $\theta\Delta$ after suitable choice of $j_m$, $j_\alpha$ or $\bar{u}$.
 We postulate the Fourier's law for $q$ and invoke the maximizing entropy principle to determine $j$.
 For Type I, our method also yields \eqref{EICHE}.
 We also derive Type I\hspace{-0.1em}I and Type I\hspace{-0.1em}I\hspace{-0.1em}I based on our classical method.

\vspace{0.5em}\noindent
\textbf{Local-in-time well-posedness.}
 We next discuss the unique existence of local-in-time smooth solution for smooth initial data.
 We rewrite the equation in the evolution equation of $w$ and $\theta$.
 It turns out that the equation \eqref{EICHE} becomes
\begin{equation}
    \begin{gathered} \label{EICHT}
    w_t=-\operatorname{div}j, \quad
    j=-\gamma\nabla\mu,\quad
    \mu=\delta_w \psi, \\
    \kappa\theta_t=\operatorname{div}( d\nabla\theta)+\gamma|\nabla\mu|^2
    +\theta\{\eta_w\operatorname{div}j+\eta_{\nabla w}\cdot\nabla\operatorname{div}j\}.
    \end{gathered}
\end{equation}
We assume that
\[
    \psi(w,\theta,\nabla w)
    =\psi_0(w,\theta)
    +\frac{\varepsilon(w,\theta)}{2}|\nabla w|^2, \quad
    \varepsilon>0.
\]
Here $\kappa=-\theta\psi_{\theta\theta}$ denotes the volume specific heat contents.
 We consider this problem either when $\Omega$ is a smooth bounded domain with boundary condition
\begin{equation} \label{EINB}
    \frac{\partial w}{\partial\nu}=0, \quad
    \frac{\partial\mu}{\partial\nu}=0
    \quad\text{and}\quad
    \frac{\partial\theta}{\partial\nu}=0
    \quad\text{on}\quad \partial\Omega
\end{equation}
or when $\Omega$ is $\mathbb{T}^n$ (periodic boundary conditions).
 Here $\partial/\partial\nu$ denotes the normal derivative.
 Note that because of the boundary condition \eqref{EINB}, total energy $\epsilon$ is conserved during time evolution provided that $\psi_{\nabla w}$ is parallel to $\nabla w$.

It seems that there had been no solvability result even if $\gamma$, $\kappa$, $d$, $\varepsilon$ are positive constants.
 In this paper we observe that \eqref{EICHT} is a quasilinear parabolic problem when $\eta_{\nabla w}=0$ and that an abstract framework \cite[Theorem~5.1.1]{PS} is applicable at least when $\gamma$, $\kappa$, $d$, $\varepsilon$ are positive constants.
 We thus obtain a unique local-in-time solution provided that $d/\gamma$ is large but this condition depends on initial data.
 More precisely, we require that $\eta_w(w_0,\theta_0)^2\theta_0\gamma/d$ is small for initial data $(w_0,\theta_0)$.
 We are forced to assume this condition to derive necessary maximal regularity for a linearized operator. 
 This is a perturbation type result reflecting structure of block matrix operators discussed in \cite{AH}; see Remark~\ref{RSPer}.
 If maximal regularity theory of mixed order developed by \cite{DK} would apply in our setting, such restrictions could be removed but we do not pursue this problem in this paper.
 Note that maximal regularity for elliptic operators discussed in \cite{PS} is not applicable since our linearized operator is of mixed order type.
 We note that $\eta_{\nabla w}=0$ if $\varepsilon$ is a constant.
 Unfortunately, our theory so far does not apply to the case where $\varepsilon$ depends on $\theta$ because the term $\eta_{\nabla w}\cdot\nabla\operatorname{div}j$ contains very high order term.

In \cite{DLSS}, $\varepsilon=\varepsilon_1\theta$ with positive constant $\varepsilon_1$ and
\[
    \psi_0(w,\theta)
    =\frac{(w^2-1)^2}{4}
    +c(\theta)w^2, \quad
    c(\theta)=\frac{(\theta-\bar{\theta})^3}{3}
\]
with some constant $\bar{\theta}>0$.
 This setting is excluded from our analysis since $\eta_{\nabla w}=-\varepsilon_1\nabla w$.
 In \cite{DLSS}, they introduce the regularizing term $a\nabla\partial_tw$ and construct a local-in-time solution when initial data $(w_0,\theta_0)$ is close to $(0,\bar{\theta})$ in some sense.
 The local-in-time solvability for \eqref{EICHT} is still an open problem when $\varepsilon$ depends on $\theta$ even when $\gamma$, $\kappa$, $d$ are positive constants.

\vspace{0.5em}\noindent
\textbf{Stability of an equilibrium.}
 We next discuss some stability result of an equilibrium under the evolution equation \eqref{EICHT}.
 We say that $(w,\theta)$ is an \emph{equilibrium} if $(w,\theta)$ is a time-independent solution of \eqref{EICHT} and \eqref{EINB}.
 Under the Neumann type boundary condition \eqref{EICHT} and \eqref{EINB}, it is easy to see (cf.\ Section~\ref{SSEE}) that the temperature $\theta$ and the chemical potential $\mu$ must be a constant for an equilibrium.
 In other words, an equilibrium $(w^e,\theta^e)$ must satisfy
\[
    \psi_{0w}(w^e,\theta^e)+\frac{\varepsilon_w}{2}|\nabla w|^2
    -\operatorname{div}(\varepsilon\nabla w^e)
    =\mathrm{const}, \quad
    \theta^e=\mathrm{const}
\]
with \eqref{EINB}.
 The left-hand side of the first equation equals $(\delta_w\psi)(w^e,\theta^e)$.
 From variational point of view, for a given constant $\theta^e$, $w^e$ is characterized as a stationary point of total free energy
\[
    \int_\Omega \psi(w,\theta^e)\, dx
\]
under the average control
\[
    \frac{1}{|\Omega|}\int_\Omega w\,dx=w_{\mathrm{av}},
\]
where $w_{\mathrm{av}}$ is given and $|\Omega|$ denotes the volume of $\Omega$.
 We say that $w^e$ is \emph{energetically stable} at temperature $\theta^e$ if
\[
    \int_\Omega \left\{\varepsilon|\nabla y|^2
    +\psi_{0ww}(w^e,\theta^e)|y|^2\right\} \,dx\ge0
\]
for all $y\in H^1(\Omega)$ satisfying $\int_\Omega y\,dx=0$.
 This stability condition can be interpreted as the non-negativity of the second variation of $\int\psi(w,\theta^e)\,dx$ under the average control.

We next discuss the condition of non-degeneracy.
 Let $\mathcal{E}_{\theta^e}^{w_\mathrm{av}}$ be a set of $w^e$ where $(w^e,\theta^e)$ is an equilibrium satisfying $\int_\Omega w^e\,dx=|\Omega|w_\mathrm{av}$.
 For $w^e\in\mathcal{E}_{\theta^e}^{w_\mathrm{av}}$ let $B^e$ be the linearized operator defined by
\[
    B^e =P_1\left(-\varepsilon\Delta+\psi_{0ww}(w^e,\theta^e)\right)
    \quad\text{with}\quad
    P_1g=g-\frac{1}{|\Omega|}\int_\Omega g\,dx
\]
in $L_\mathrm{av}^2(\Omega)$, where $L_\mathrm{av}^2(\Omega)$ is the average-free $L^2$ space.
 We say that $w^e$ is \emph{non-degenerate} if $\mathcal{E}_{\theta^e}^{w_{\mathrm{av}}}$ is an $m$-dimensional manifold near $w^e$ and $m=\dim\ker B^e$, where $\dim$ denotes the dimension and $\ker$ denotes the kernel.
 This wording is the same as in \cite{GKTs}.
 Indeed, in \cite{GKTs} for the isothermal Navier-Stokes-Korteweg flow it was shown that any solution starting from near equilibrium extends globally-in-time and it converges to some equilibrium near original equilibrium if the original equilibrium is energetically stable and non-degenerate for $1\le n\le3$ by applying the generalized stability principle \cite[Theorem~5.3.1]{PS}.  

In our setting, the spectral analysis of the linearized operator around $(w^e,\theta^e)$ of \eqref{EICHT} is more involved and we are able to prove above type of stability of an equilibrium only when $\theta^e=\kappa^2\gamma/d$ and $\nabla\eta_w(w^e,\theta^e)\equiv0$ by applying \cite[Theorem~5.3.1]{PS}.
 Although our setting is rather restrictive, it seems that this is the first global and stability result.

\vspace{0.5em}\noindent
\textbf{A brief literature overview.}
 There are several models of temperature-dependent Cahn-Hilliard equations.
 Since \cite{DLSS} contains a nice review, our explanation is very short except articles not mentioned there.
 A popular model by G.~Caginalp \cite{Ca} may not satisfy the conservation of total energy so it is outside our framework.
 A special case of the models derived by \cite{AP} corresponds to Type I\hspace{-0.1em}I, i.e., $j_m=0$.
 H.~Freist\"uhler and M.~Kotschote \cite{FK1} derive the Navier-Stokes-Cahn-Hilliard equation assuming $j_m=0$.
 If fluid effect is neglected, their model corresponds to Type I\hspace{-0.1em}I.
 They further derive another model in \cite{FK2} assuming that $j_\alpha=0$.
 They obtain Type I model which agrees with ours.
 In these derivations, Fourier's law is assumed and the form of flux $j$ is assumed to be a gradient of some known function called a chemical potential.
 The model derived by A.~Miranville and G.~Schimperna \cite{MS} takes $j=-\nabla(\mu/\theta)$, $\mu=\delta_w\psi$ as a special case which is still Type I but different choice of $q$, i.e., $q+\mu j=-d\nabla\theta$.
 We shall discuss the derivation of this model from our point of view in Remark~\ref{RMir} in Section~\ref{SSEq}.
 In \cite{FK1}, \cite{FK2}, the fluid effect is considered.
 In the isothermal case, the Cahn-Hilliard equation with fluid effect called the Navier-Stokes-Cahn-Hilliard equation is first derived by J.~Lowengrub and L.~Truskinovsky \cite{LT} for compressible fluid.
 For the incompressible case, the Navier-Stokes-Cahn-Hilliard equation is derived by H.~Abels, H.~Garcke and G.~Gr\"un \cite{AGG}.

Let us mention well-posedness of the equations.
 For isothermal case, this type of problems is well studied for various choices of $\gamma(w)$ and $\psi(w,\nabla w)$.
 See \cite{Mi} and recent paper \cite{CGR} and papers cited there.
 However, for non-isothermal model it seems that very little is known.
 In \cite{MaS}, A.~Marveggio and G.~Schimperna proved the existence of a global-in-time weak and entropy solutions for a subclass of equations derived in \cite{MS}.
 (Note that in their definition, notion of a weak solution is closer to that of a classical solution than that of an entropy solution.)
 However, it seems that its uniqueness is not known.
 Our result for \eqref{EICHT} seems to be the first result for unique solvability of the time-dependent Cahn-Hilliard type equation.
 Also it seems that the stability of equilibrium was not discussed in the literature for non-isothermal case.

This paper is organized as follows.
 In Section~\ref{SD}, we derive several Cahn-Hilliard equations from conservation law and the second law of thermodynamics.
 In Section~\ref{SE}, we rewrite equations in the form of evolution equations of mass fraction $w$ and temperature $\theta$.
 In Section~\ref{SL}, to construct solutions for initial-boundary value problems we introduce function spaces and write equations in an abstract way.
 In Section~\ref{SW}, we construct a local-in-time solution with no size restriction of initial data.
 In Section~\ref{SS}, we discuss the stability of equilibria.

\section{Derivation} \label{SD}

We recall the derivation of temperature-dependent Cahn-Hilliard equation.
 There seem to be several ways to derive and the resulting system is different depending on what kind of a priori assumptions on the system are imposed.

We begin with several conservation laws (cf.\ \cite[Chapter~1]{PS}).
 Let $\rho=\rho(x,t)$ be the density of a mixed fluid at $x\in\Omega$ and at time $t>0$, where $\Omega$ is a domain in $\mathbb{R}^n$.
 Let $u=(u^1,\ldots,u^n)$, $u^i=u^i(x,t)$ be its velocity.
 The fluid consists of two components, material one and two.
 Let $w=w(x,t)$ denote the mass fraction of the material one so that $\rho w=\rho_1$ is its apparent density while $\rho(1-w)=\rho_2$ is the apparent density of the material two.
 Evidently $\rho_1+\rho_2=\rho$.
 Physically, we require $\rho_1\ge0$, $\rho_2\ge0$ so that $0\le w\le1$.
 The equations of conservation of mass are
\begin{align}
    &\partial_t\rho+\operatorname{div}(\rho u)=0, \label{EM} \\
    &\partial_t\rho_1+\operatorname{div}(\rho_1 u)=-\operatorname{div}j, \label{EM1} \\
    &\partial_t\rho_2+\operatorname{div}(\rho_2 u)=\operatorname{div}j, \label{EM2}
\end{align}
where $j$ denotes the \emph{mass flux}.
 The material one and two may diffuse but no chemical reaction is considered.
 Of course, \eqref{EM} and \eqref{EM1} imply \eqref{EM2} so we only consider \eqref{EM} and \eqref{EM1}.
 If we use the mass fraction $w$ \eqref{EM1} becomes
\begin{equation} \label{EMw}
    \partial_t(\rho w)+\operatorname{div}(\rho wu)=-\operatorname{div}j.
\end{equation}
Let $T$ be the stress tensor.
 The equation of conservation of momentum is of the form if there is no external body force:
\begin{equation} \label{EMo}
    \partial_t(\rho u)+\operatorname{div}(\rho u\otimes u)-\operatorname{div}T=0.  
\end{equation}
Let $\psi$ be the mass specific Helmholtz energy.
 Let $\eta$ be the mass specific entropy.
 Let $\epsilon$ be the mass specific internal energy.
 It is defined as
\begin{equation} \label{EE}
    \epsilon=\psi+\eta\theta,
\end{equation}
where $\theta=\theta(x,t)$ denotes the temperature.
 Physically, we require $\theta>0$.
 Since $\rho\left(\epsilon+|u|^2/2\right)$ is the density of total energy, the conservation of energy is of the form
\begin{equation} \label{EEC}
    \partial_t\left(\rho\left(\epsilon+\frac{|u|^2}{2}\right)\right)+\operatorname{div}\left(\rho u\left(\epsilon+\frac{|u|^2}{2}\right)\right)-\operatorname{div}Tu-\operatorname{div}j_e=0,
\end{equation}
where $j_e$ denotes the \emph{energy flux}.
 Here, we assume that there is no heat source.
 One has to consider the second law of thermodynamics.
 This can be written as follows:
\begin{equation} \label{EEP}
    (\rho\eta)_t+\operatorname{div}(\rho u\eta)+\operatorname{div}j_s=:\Delta
\end{equation}
and we require that the \emph{entropy production} $\xi=\theta\Delta$ is non-negative with some entropy flux $j_s$.

To derive equations for $\rho$, $u$, $\theta$ and $w$ one has to close the system by introducing several constitutive relations.
 In this paper, we do not consider fluid effect so that we assume $u\equiv0$ and disregard the momentum conservation \eqref{EMo}.
 Moreover, we assume that the density is constant independent of $x$ and $t$.
 We assume that $\rho\equiv1$.
 Then the equation \eqref{EMw} becomes
\begin{equation} \label{EW}
    \partial_t w=-\operatorname{div}j
\end{equation}
and \eqref{EEC} becomes
\begin{equation} \label{EEC1}
    \partial_t \epsilon=\operatorname{div}j_e.
\end{equation}
The second law of thermodynamics is now of the form
\begin{equation} \label{EEP1}
    0\le\Delta=\partial_t\eta+\operatorname{div}j_s.
\end{equation}

Let $\psi_0=\psi_0(w,\theta)$ be a bulk free energy.
 Since $\rho$ is assumed to be $1$, it is essentially the mass specific Helmholtz energy $\psi$.
 However, if $\psi_0$ is non-convex in $w$, then the resulting equation for $\theta$ has a part of backward parabolicity so its initial value problem may not be solvable for a general smooth initial data.
 To make the problem well-posed, we consider the Helmholtz energy $\psi$ of the form
\begin{equation} \label{EH1}
    \psi=\psi_0(w,\theta)+\frac{\varepsilon(w,\theta)}{2}|\nabla w|^2
\end{equation}
with $\varepsilon>0$ for $(w,\theta)$ at least where $\psi_0$ is non-convex in $w$.
 More generally we consider
\begin{equation} \label{EH}
    \psi=\psi(w,\theta,\nabla w).
\end{equation}

The goal of this section is to derive a closed system form \eqref{EW}, \eqref{EEC1}, \eqref{EEP1} when $\psi=\psi(w,\theta,\nabla w)$ by choosing several ``constitutive relations" for $j$, $j_e$ and $j_s$.
 We recall that the entropy satisfies $\eta=-\partial_\theta\psi$.
 Let $q$ be the heat flux.
 We assume Fourier's law:
\begin{equation} \label{EF}
    q=-d(w,\theta)\nabla\theta.
\end{equation}
We consider typical choices of $j_e$ and $j_s$,
\begin{description}
    \item[\rm Type I] $j_s=q/\theta$ while $j_e=-q+j_m$ with some flux $j_m$;
    \item[\rm Type I\hspace{-0.1em}I] $j_e=-q$ while $j_s=q/\theta+j_\alpha$ with some flux $j_\alpha$.
\end{description}
There is another choice
\begin{description}
    \item[\rm Type I\hspace{-0.1em}I\hspace{-0.1em}I] $j=\bar{u}w$, $j_s=\bar{u}\eta+q/\theta$ with some vector field $\bar{u}$ while $j_e=-q+j_m$ with some flux $j_m$. 
\end{description}
In Type I\hspace{-0.1em}I\hspace{-0.1em}I, the entropy diffuses along diffusion of the material one \cite{DLSS}.

\subsection{Entropy production} \label{SSEP}

We begin with calculating entropy production $\xi=\theta\Delta$, where $\Delta$ is defined by \eqref{EEP1}.
 We consider \eqref{EW}, \eqref{EEC1} with \eqref{EH}, where $\epsilon=\psi+\eta\theta$ and $\eta=-\partial_\theta\psi$.
 We denote $\psi_w=\partial\psi/\partial w$, $\psi_{\nabla w}=\left.\nabla_z\psi(w,\theta,z)\right|_{z=\nabla w}$ and
\[
    \delta_w\psi:=\psi_w-\operatorname{div}\psi_{\nabla w},
\]
which is an Euler-Lagrange operator of $\int_\Omega\psi(w,\theta,\nabla w)\,dx$ (with respect to $L^2$-inner product).
\begin{thm} \label{TT}
\begin{enumerate}
    \item[(i)] (Type I)
    Let $\mu$ denote the chemical potential of the form $\mu:=\delta_w\psi$ ($=\psi_w-\operatorname{div}\psi_{\nabla w}$).
     Assume that $j_s=q/\theta$ and 
    \[
        j_e=-q+j_m
    \]
    with
    \[
        j_m=-\left\{\mu j+\psi_{\nabla w}\operatorname{div}j \right\}.
    \]
    Then
    \[
        \xi =-\nabla\mu\cdot j-q\cdot\nabla\theta/\theta.
    \]
    \item[(i\hspace{-0.1em}i)] (Type I\hspace{-0.1em}I)
    Let $\check{\mu}$ denote the chemical potential of the form $\check{\mu}:=\theta\delta_w(\psi/\theta)$.
    Assume that $j_e=-q$ and 
    \[
        j_s=q/\theta+j_\alpha
    \]
    with
    \[
        j_\alpha=-\left\{\check{\mu}j/\theta+(\psi_{\nabla w}/\theta)\operatorname{div}j \right\}.
    \]
    Then
    \[
        \xi=-\theta\nabla(\check{\mu}/\theta)\cdot j-q\cdot\nabla\theta/\theta.
    \]
    \item[(i\hspace{-0.1em}i\hspace{-0.1em}i)] (Type I\hspace{-0.1em}I\hspace{-0.1em}I) Assume that $j=\bar{u}w$ and $j_s=\bar{u}\eta+q/\theta$ for some vector field $\bar{u}$ and that $j_e=-q+j_m$ with
     \[
        j_m=-\left\{\mu j+\psi_{\nabla w}\operatorname{div}j+\frac{\theta\eta}{w}j \right\}.
     \]
     Then
     \[
        \xi=-\left(\nabla \mu+\frac{\eta\nabla\theta}{w}\right) \cdot j-q\cdot\frac{\nabla\theta}{\theta}.
     \]
\end{enumerate}
\end{thm}
\begin{proof}
To calculate variation $\partial_t\epsilon$, it is convenient to introduce ``isothermal variation" $V_\mathrm{ISO}$ of $\psi$ of the form 
    \begin{equation} \label{EBISO}
        V_\mathrm{ISO}:= \psi_w w_t+\psi_{\nabla w}\cdot\nabla w_t
        \quad\text{with}\quad w_t=\partial_t w.
    \end{equation}
    We proceed
    \begin{equation}
    \begin{aligned} \label{EISO}
        \partial_t\epsilon
        &=\psi_\theta \theta_t+V_\mathrm{ISO}+\eta\theta_t+(\partial_t\eta)\theta \\
        &=V_\mathrm{ISO}+\theta\partial_t\eta,
    \end{aligned}
    \end{equation}
    since $\eta=-\partial_\theta\psi$.
     Since $\delta_w\psi=\psi_w-\operatorname{div}\psi_{\nabla w}$, we have
    \begin{align*}
        V_\mathrm{ISO}
        &=\psi_w w_t+\operatorname{div}(\psi_{\nabla w}w_t)-(\operatorname{div}\psi_{\nabla w})w_t \\
        &=(\delta_w\psi)w_t+\operatorname{div}(\psi_{\nabla w}w_t).
    \end{align*}
    Since $w_t=-\operatorname{div}j$ by \eqref{EW}, we observe that
    \begin{equation} \label{EVI1}
        V_\mathrm{ISO}
        =-(\delta_w\psi)\operatorname{div}j-\operatorname{div}(\psi_{\nabla w}\operatorname{div}j).
    \end{equation}
    Since $\partial_t\epsilon=\operatorname{div}j_e$ by \eqref{EEC1} so that $\theta\partial_t\eta=\operatorname{div}j_e-V_\mathrm{ISO}$ by \eqref{EISO} and since \eqref{EEP1} becomes $\theta\partial_t\eta=\xi-\theta\operatorname{div}j_s$, we see that
    \begin{equation} \label{ERE}
        \xi=\operatorname{div}j_e+\theta\operatorname{div}j_s-V_\mathrm{ISO}.
    \end{equation}
\begin{enumerate}
    \item[(i)] (Type I) 
    Since $(\delta_w\psi)\operatorname{div}j=\operatorname{div}\left((\delta_w\psi)j\right)-\nabla(\delta_w\psi)\cdot j$, from \eqref{EVI1} it follows that 
    \begin{equation} \label{EVI}
        V_\mathrm{ISO}=-\operatorname{div}\left((\delta_w\psi)j\right)+\nabla(\delta_w\psi)\cdot j
        -\operatorname{div}(\psi_{\nabla w}\operatorname{div}j).
    \end{equation}
    By our assumptions $j_s=q/\theta$ and $j_e=-q+j_m$, \eqref{ERE} together with \eqref{EVI} yields
    \begin{align*}
        \xi
        =&-\operatorname{div}q+\theta\operatorname{div}(q/\theta)
        +\operatorname{div}\left(j_m
        +(\delta_w\psi)j+\psi_{\nabla w}\operatorname{div}j\right) \\
        &-\nabla(\delta_w\psi)\cdot j.
    \end{align*}
    By our choice of $j_m$, we observe that
    \[
        \xi=-\operatorname{div}q+\theta\operatorname{div}(q/\theta)-\nabla(\delta_w\psi)\cdot j.
    \]
    Since
    \[
        -\operatorname{div}q+\theta\operatorname{div}(q/\theta)
        =-q\cdot\nabla\theta/\theta,
    \]
    the conclusion follows.
    \item[(i\hspace{-0.1em}i)] (Type I\hspace{-0.1em}I)
    We write $V_\mathrm{ISO}$ in a different way from \eqref{EVI1}.
     We set $\varphi=\psi/\theta$.
      By \eqref{EBISO}, we see that
      \begin{equation} \label{EBISO1}
     V_\mathrm{ISO}=\theta\{\varphi_w w_t+\varphi_{\nabla w}\cdot\nabla w_t\}
      \end{equation}
      since $\psi_w=\theta\varphi_w$ and $\psi_{\nabla w}=\theta\varphi_{\nabla w}$.
      In a parallel way to derive \eqref{EVI1} from \eqref{EBISO}, we obtain
      \[
     V_\mathrm{ISO}=-\theta\left\{(\delta_w\varphi)\operatorname{div}j+\operatorname{div}(\varphi_{\nabla w}\operatorname{div}j)\right\}.
     \]
     In a parallel way to derive \eqref{EVI} from \eqref{EVI1} we obtain
     \[
     -V_\mathrm{ISO}=\theta\left\{\operatorname{div}\left((\delta_w\varphi)j\right)-\nabla(\delta_w\varphi)\cdot j+\operatorname{div}(\varphi_{\nabla w}\operatorname{div}j)\right\},
     \]
     which is an expression different from \eqref{EVI}.
      Plugging $j_e=-q$ and $j_s=q/\theta+j_\alpha$ to \eqref{ERE}, we obtain
    \begin{align*}
        \xi
        &=-\operatorname{div}q+\theta\operatorname{div}(q/\theta)+\theta\operatorname{div}j_\alpha-V_\mathrm{ISO} \\
        &=-q\cdot\nabla\theta/\theta+\theta\left\{\operatorname{div}j_\alpha+\operatorname{div}\left((\delta_w\varphi)j+\varphi_{\nabla w}\operatorname{div}j\right)\right\}-\theta\nabla\delta_w\varphi\cdot j \\
        &=-q\cdot\nabla\theta/\theta-\theta\nabla\left(\delta_w(\psi/\theta)\right)\cdot j
    \end{align*}
    if we choose $j_\alpha$ as in Theorem~\ref{TT}.
     The conclusion now follows.
     \item[(i\hspace{-0.1em}i\hspace{-0.1em}i)] (Type I\hspace{-0.1em}I\hspace{-0.1em}I)
      We first note that
      \[
        j_s=\frac{q}{\theta}+\frac{\eta}{w}j.
      \]
       The identity \eqref{ERE} together with \eqref{EVI} yields
       \begin{align*}
        \xi=&\operatorname{div}(-q+j_m)+\theta\left(\operatorname{div}(\eta j/w+q/\theta)\right) \\
        &+(\delta_w\psi)\operatorname{div}j+\operatorname{div}(\psi_{\nabla w}\operatorname{div}j) \\
        =&\left\{\operatorname{div}(-q)+\theta\operatorname{div}(q/\theta)\right\}
        +\operatorname{div}(j_m+\psi_{\nabla w}\operatorname{div}j) \\
        &+(\delta_w\psi)\operatorname{div}j+\theta\operatorname{div}(\eta j/w).
       \end{align*}
       Plugging $(\delta_w\psi)\operatorname{div}j=\operatorname{div}\left((\delta_w\psi)j\right)-\nabla(\delta_w\psi)\cdot j$ and $\theta\operatorname{div}(\eta j/w)=\operatorname{div}(\theta\eta j/w)-\nabla\theta\cdot\eta j/w$ to the previous identity, we have
       \begin{align*}
        \xi=&-q\cdot(\nabla\theta)/\theta+\operatorname{div}(j_m+\psi_{\nabla w}\operatorname{div}j) \\
        &+\left\{\operatorname{div}\left((\delta_w\psi)j\right)-\nabla(\delta_w\psi)\cdot j \right\} \\
        &+\left\{\operatorname{div}(\theta\eta j/w)-(\nabla\theta)\cdot\eta j/w\right\} \\
        =&-q\cdot(\nabla\theta)/\theta+\operatorname{div}\left(j_m+\psi_{\nabla w}\operatorname{div}j+(\delta_w\psi)j+\theta\eta j/w \right) \\
        &-\left\{\nabla(\delta_w\psi)+(\nabla\theta)\eta/w\right\}\cdot j.
       \end{align*}
       By our choice of $j_m$, we observe that
       \[
        \xi=-q\cdot(\nabla\theta)/\theta -\left(\nabla(\delta_w\psi)+\frac{\eta\nabla\theta}{w}\right) \cdot j.
       \]
       As in Type I, this yields what we would like to prove. 
\end{enumerate}
\end{proof}

If we admit Fourier's law \eqref{EF}, Theorem~\ref{TT} says that the entropy production is of the form
\begin{align}
     &\xi=\theta\Delta=d|\nabla\theta|^2/\theta-\nabla\mu\cdot j, \quad\text{(Type I),} \label{EPE1} \\
     &\xi=\theta\Delta=d|\nabla\theta|^2/\theta-\theta\nabla(\check{\mu}/\theta)\cdot j, \quad\text{(Type I\hspace{-0.1em}I),} \label{EPE2} \\
     &\xi=\theta\Delta=d|\nabla\theta|^2/\theta-\left(\nabla\mu+\eta\nabla\theta/w\right)\cdot j, \quad\text{(Type I\hspace{-0.1em}I\hspace{-0.1em}I).} \label{EPE3}
\end{align}
\begin{prop} \label{PH}
    Assume that $\theta>0$.
     Then $d\ge0$ provided that $\Delta\ge0$ for $j=0$.
\end{prop}
This is clear by \eqref{EPE1}, \eqref{EPE2}, \eqref{EPE3}.
 Physically, it says that the conductivity $d$ must be non-negative provided that the second law of thermodynamics holds.

In the sequel we always assume that $\theta>0$.
 To determine $j$ we invoke what is called maximizing entropy principle as discussed in \cite{HM}; see also \cite{MP}.
 Here is a variant
\begin{center}
 ``Maximize $\xi$ keeping $|j|=$\;constant (by fixing $\theta$)."    
\end{center}
By the equality case of the Schwarz inequality, $j$ must be
\begin{align*}
    &j=-\gamma\nabla\mu \  \text{(Type I)}, \quad
    j=-\tilde{\gamma}\nabla(\check{\mu}/\theta) \ \text{(Type I\hspace{-0.1em}I)}, \\
    &j=-\tilde{\tilde{\gamma}}\left(\nabla\mu+\eta(\nabla\theta)/w\right) \ \text{(Type I\hspace{-0.1em}I\hspace{-0.1em}I)},
\end{align*}
where $\gamma$, $\tilde{\gamma}$ or $\tilde{\tilde{\gamma}}$ are non-negative functions of $w$ and $\theta$.
 Then by Theorem~\ref{TT} we see that
\begin{align*}
     &\xi=d|\nabla\theta|^2/\theta+\gamma|\nabla\mu|^2 \quad\text{(Type I),} \\
     &\xi=d|\nabla\theta|^2/\theta+\tilde{\gamma}\theta\left|\nabla(\check{\mu}/\theta)\right|^2 \quad\text{(Type I\hspace{-0.1em}I),} \\
    &\xi=d|\nabla\theta|^2/\theta+\tilde{\tilde{\gamma}}\left|\nabla\mu+\eta(\nabla\theta)/w\right|^2 \quad\text{(Type I\hspace{-0.1em}I\hspace{-0.1em}I).}
\end{align*}
The coefficients $\gamma$, $\tilde{\gamma}$ or $\tilde{\tilde{\gamma}}$ are often called mobilities.
\begin{rem} \label{RDer}
In our Theorem~\ref{TT} we give $j_m$ and $j_\alpha$ as prescribed quantities.
 However, these relations for $j_m$ and $j_\alpha$ are actually derived up to divergence free vector provided that $\xi\ge0$ for all choice of functions $j$ assuming $q=0$.
\end{rem}

\subsection{Equations} \label{SSEq}

Let us summarize the equation we obtain.
 We postulate Fourier's law. \\
Type I.
\[
    \left\{
    \begin{array}{l}
	\partial_t w=-\operatorname{div}j,\quad j=-\gamma\nabla\mu \\
	\partial_t\epsilon=\operatorname{div}(d\nabla\theta)-\operatorname{div}(\mu j+\psi_{\nabla w}\operatorname{div}j).
    \end{array}
    \right.
\]
Here $\mu=\delta_w\psi$ and $\gamma\ge0$, $d\ge0$.
 The quantity $\Delta=\xi/\theta$ for entropy production $\xi$ equals
\[
    \Delta=d|\nabla\theta|^2/\theta^2
    +\gamma|\nabla\mu|^2/\theta.
\]
The entropy flux is $j_s=-d(\nabla\theta)/\theta$.\\
Type I\hspace{-0.1em}I.
\[
    \left\{
    \begin{array}{l}
	\partial_t w=-\operatorname{div}j,\quad j=-\tilde{\gamma}\nabla(\check{\mu}/\theta), \\
	\partial_t\epsilon=\operatorname{div}(d\nabla\theta).
    \end{array}
    \right.
\]
Here $\check{\mu}=\theta\delta_w(\psi/\theta)$ and $\tilde{\gamma}\ge0$, $d\ge0$.
 The quantity $\Delta$ is
\[
    \Delta=d|\nabla\theta|^2/\theta^2
    +\tilde{\gamma}\left|\nabla(\check{\mu}/\theta)\right|^2.
\]
The entropy flux is
\[
    j_s=-d(\nabla\theta)/\theta
    -\left\{\check{\mu} j/\theta+(\psi_{\nabla w}/\theta)\operatorname{div}j \right\}.
\]
Type I\hspace{-0.1em}I\hspace{-0.1em}I.
\[
    \left\{
    \begin{array}{l}
	\partial_t w=-\operatorname{div}j,\quad j=-\tilde{\tilde{\gamma}}\left(\nabla\mu+\eta(\nabla\theta)/w\right), \\
	\partial_t\epsilon=\operatorname{div}(d\nabla\theta)-\operatorname{div}\left(\mu j+\psi_{\nabla w}\operatorname{div}j+\theta\eta j/w\right).
    \end{array}
    \right.
\]
The quantity $\Delta$ equals
\[
    \Delta= d|\nabla\theta|^2/\theta^2+\tilde{\tilde{\gamma}}\left|\nabla\mu+\eta\nabla\theta/w\right|^2/\theta.
\]
The entropy flux is
\[
    j_s=-d(\nabla\theta)/\theta
    -\tilde{\tilde{\gamma}}\left\{\nabla\mu+\eta(\nabla\theta)/w\right\}\eta/w.
\]
\begin{rem} \label{RMir}
In \cite{MS} there are various choices of $j$.
 For example, in one of the models they take $j=-\nabla(\mu/\theta)$.
 Their model is Type I but Fourier's law is replaced by another law to derive equations compatible with the second law of thermodynamics.
 Let us calculate $\xi$ in Theorem~\ref{TT} (i).
 By this choice of $j$, we have
\[
    \xi=\nabla\mu\cdot\nabla\left(\frac{\mu}{\theta}\right)
    -q\cdot \frac{\nabla\theta}{\theta}.
\]
Since $\nabla\mu=\theta\nabla(\mu/\theta)+\frac{\mu}{\theta}\nabla\theta$, we see that
\[
    \xi=\theta\left|\nabla\left(\frac{\mu}{\theta}\right)\right|^2
    -\frac{\nabla\theta}{\theta}\cdot(\mu j+q).
\]
If we set
\[
    -q-\mu j=d\nabla\theta
\]
instead of Fourier's law, then $\xi\ge0$ if and only if $d\ge0$.
 The resulting equation is
\[
    \left\{
    \begin{array}{l}
	\partial_t w=-\operatorname{div}j,\quad j=-\nabla(\mu/\theta) \\
	\partial_t\epsilon=-\operatorname{div}q-\operatorname{div}(\mu j+\psi_{\nabla w}\operatorname{div}j) \\
    -q-\mu j=d\nabla\theta.
    \end{array}
    \right.
\]
The entropy production is
\[
    \xi=\theta\left|\nabla\left(\frac{\mu}{\theta}\right)\right|^2
    +d\frac{|\nabla\theta|^2}{\theta}.
\]
\end{rem}

\section{Evolution law of temperature and equilibrium} \label{SE}

In this section we derive equations for temperature $\theta$ and mass fraction $w$ to see the structure of the Cahn-Hilliard equations.
 We also discuss an equilibrium of the system.

\subsection{Temperature evolution} \label{SSTE}

We shall rewrite
\[
    \partial_t\epsilon=\operatorname{div}(d\nabla\theta)+\operatorname{div}j_m
\]
in the form of evolution of temperature $\theta$.
 We calculate $\partial_t\epsilon$ in a different way as follows.
\[
    \partial_t\epsilon=\kappa\partial_t\theta+V_\mathrm{ISO}
    +\theta(\eta_ww_t+\eta_{\nabla w}\cdot\nabla w_t),
\]
where $\kappa=\epsilon_\theta$ is the volume specific heat content.
 Since $\eta=-\psi_\theta$,
\[
    \kappa=\psi_\theta+\eta+\theta\eta_\theta=-\theta\psi_{\theta\theta}.
\]
By \eqref{EVI}, we see
\[
    V_\mathrm{ISO}=-\operatorname{div}\left((\delta_w\psi)j+\psi_{\nabla w}\operatorname{div}j\right)
    +\nabla(\delta_w\psi)\cdot j.
\]
Since
\[
    \eta_w w_t+\eta_{\nabla w}\cdot \nabla w_t
    =-\eta_w\operatorname{div}j-\eta_{\nabla w}\nabla\operatorname{div}j,
\]
we conclude that
\begin{equation}
   \begin{split} \label{ETEv}
    \kappa\partial_t\theta
    = &\partial_t\epsilon+\operatorname{div}\left((\delta_w\psi)j+\psi_{\nabla w}\operatorname{div}j\right)-\nabla(\delta_w\psi)\cdot j \\
    &+\theta\{\eta_w\operatorname{div}j+\eta_{\nabla w}\cdot\nabla\operatorname{div}j\}.
\end{split} 
\end{equation}
In the case of Type I we know that $\partial_t\epsilon=\operatorname{div}(d\nabla\theta)+\operatorname{div}j_m$ with
\[
    j_m=-\left\{(\delta_w\psi)j+\psi_{\nabla w}\operatorname{div}j \right\}.
\]
The identity \eqref{ETEv} now yields
\[
    \kappa\partial_t\theta=\operatorname{div}(d\nabla\theta) -\nabla(\delta_w\psi)\cdot j
    +\theta\{\eta_w\operatorname{div}j+\eta_{\nabla w}\cdot\nabla\operatorname{div}j\}.
\]
A similar calculation yields evolution equations of temperature for Type I\hspace{-0.1em}I and Type I\hspace{-0.1em}I\hspace{-0.1em}I.
 Let us summarize.
\begin{prop} \label{PReT}
    \begin{enumerate}
        \item[(i)] In the case of Type I, the Cahn-Hilliard equations are of the form
        \begin{align}
        \begin{split} \label{EETI}
            w_t= &-\operatorname{div}j, \quad j=-\gamma\nabla\mu, \quad \mu=\delta_w\psi, \\
            \kappa\theta_t= &\operatorname{div}(d\nabla\theta)-\nabla(\delta_w\psi)\cdot j \\
            & +\theta \left\{\eta_w\operatorname{div}j+\eta_{\nabla w}\cdot\nabla\operatorname{div}j \right\},
        \end{split}
        \end{align}
        where $\kappa=-\theta\psi_{\theta\theta}$.
        \item[(i\hspace{-0.1em}i)] In the case of Type I\hspace{-0.1em}I, the Cahn-Hilliard equations are of the form
        \begin{align}
        \begin{split} \label{EETII}
            w_t= &-\operatorname{div}j, \quad j=-\tilde{\gamma}\nabla(\check{\mu}/\theta), \quad \check{\mu}=\theta\delta_w(\psi/\theta) \\
            \kappa\theta_t= &\operatorname{div}(d\nabla\theta)-\nabla(\delta_w\psi)\cdot j \\
            & +\operatorname{div}\left\{(\delta_w\psi)j+\psi_{\nabla w}\operatorname{div}j\right\} \\
            & +\theta\{\eta_w\operatorname{div}j+\eta_{\nabla w}\cdot\nabla\operatorname{div}j\}.
        \end{split}
        \end{align}
    The second equation can be rewritten as
\[
    \kappa\theta_t=\operatorname{div}(d\nabla\theta)
    +\epsilon_w\operatorname{div}j+\epsilon_{\nabla w}\cdot\nabla\operatorname{div}j.
\]    \item[(i\hspace{-0.1em}i\hspace{-0.1em}i)] In the case of Type I\hspace{-0.1em}I\hspace{-0.1em}I, the Cahn-Hilliard equations are of the form
        \begin{align*}
        \begin{split}
            w_t= &-\operatorname{div}j, \quad j=-\tilde{\tilde{\gamma}}\left(\nabla\mu+\eta\nabla\theta/w\right) \\
            \kappa\theta_t= &\operatorname{div}(d\nabla\theta)-\operatorname{div}\left(\theta\eta j/w \right)-\nabla(\delta_w\psi)\cdot j \\
            & +\theta\{\eta_w \operatorname{div}j+\eta_{\nabla w} \cdot\nabla\operatorname{div}j\}.
        \end{split}
        \end{align*}
    \end{enumerate}
\end{prop}

Let us pick up the highest derivative terms.
 Since
\[
    \mu=\delta_w\psi=\psi_w-\operatorname{div}\psi_{\nabla w},
\]
it contains the second-order derivative of $w$ while the first-order derivative of $\theta$ as highest order if $\psi$ is given as \eqref{EH1}.
 This yields the following table (Table~\ref{A1}) of the number of order of highest derivative terms in \eqref{EETI} and \eqref{EETII}.
\begin{table}[h!]
  \centering
  \caption{General}
  \begin{tabular}{l|c|c}
    Type I / I\hspace{-0.1em}I / I\hspace{-0.1em}I\hspace{-0.1em}I & $w$ & $\theta$ \\ \hline
    $w$-equation & 4 & 3 \\ \hline
    $\theta$-equation & 5 & 4 \\ \hline
  \end{tabular}
  \label{A1}
\end{table}
\\ In the case of $\eta_{\nabla w}\equiv0$ (or $\partial_\theta\psi_{\nabla w}\equiv0$), it can be modified as follows since the highest order for $\theta$ in $\mu$ or $\check{\mu}$ is now $0$; see Table~\ref{A2}.
\begin{table}[h!]
  \centering
  \caption{The case $\eta_{\nabla w}=0$}
  \begin{minipage}[c]{0.25\hsize}
  \centering
  \begin{tabular}{l|c|c}
    Type I / I\hspace{-0.1em}I\hspace{-0.1em}I & $w$ & $\theta$ \\ \hline
    $w$-equation & 4 & 2 \\ \hline
    $\theta$-equation & 4 & 2 \\ \hline
  \end{tabular}
  \end{minipage}
  \begin{minipage}[c]{0.25\hsize}
    \centering
  \begin{tabular}{l|c|c}
    Type I\hspace{-0.1em}I & $w$ & $\theta$ \\ \hline
    $w$-equation & 4 & 2 \\ \hline
    $\theta$-equation & 5 & 3 \\ \hline
  \end{tabular}
  \label{A2}
    \end{minipage}
\end{table}

If $\psi$ is of the form \eqref{EH1}, $\eta_{\nabla w}=0$ is equivalent to saying that $\varepsilon$ is independent of $\theta$.
 In Sections \ref{SL}--\ref{SS}, we only consider Type I with $\eta_{\nabla w}=0$ since we do not know in what way we have to treat the term $\nabla\operatorname{div}j$ which contains 5th derivative of $w$.

\subsection{Equilibrium} \label{SSEE}

We shall study equilibrium of \eqref{EETI}, \eqref{EETII}.
 We consider these equations in a bounded domain $\Omega$ in $\mathbb{R}^n$ with smooth boundary or a periodic cell $\Omega=\mathbb{T}^n$.
 In the case of a bounded domain we impose boundary conditions
\begin{equation} \label{EBT1}
    \frac{\partial w}{\partial\nu}=0, \quad
    \frac{\partial\mu}{\partial\nu}=0
    \quad\text{and}\quad
    \frac{\partial\theta}{\partial\nu}=0
\end{equation}
for Type I and
\begin{equation} \label{EBT2}
    \frac{\partial w}{\partial\nu}=0, \quad
    \frac{\partial\check{\mu}}{\partial\nu}=0
    \quad\text{and}\quad
    \frac{\partial\theta}{\partial\nu}=0
\end{equation}
for Type I\hspace{-0.1em}I, where $\nu$ denotes the exterior unit normal.
 We say that $(w,\theta)$ ($\theta>0$) is an \emph{equilibrium} of \eqref{EETI} (resp.\ \eqref{EETII}) with \eqref{EBT1} (resp.\ \eqref{EBT2}) if $(w,\theta)$ is a time-independent solution of \eqref{EETI} (resp.\ \eqref{EETII}) with \eqref{EBT1} (resp.\ \eqref{EBT2}).
\begin{prop} \label{PST}
    Assume that $(w,\theta)$ is a smooth equilibrium of \eqref{EETI} (resp.\ \eqref{EETII}) in $\Omega$ satisfying \eqref{EBT1} (resp.\ \eqref{EBT2}).
     Assume that $\gamma$, $\tilde{\gamma}$, $d$ and $\psi$ are smooth functions with respect to their independent variables.
     Assume that $\gamma$, $\tilde{\gamma}$, $d\ge c_0$ with some positive constant $c_0$.
     Then $\mu$ (resp.\ $\check{\mu}$) is a constant and $\theta$ is also constant.
\end{prop}
\begin{proof}
    Since $w$ is time-independent, it must satisfy
\[
    \operatorname{div}(\gamma\nabla\mu)=0
    \quad\text{in}\quad \Omega
\]
for Type I.
 By \eqref{EBT1} (for $\mu$) and the strong maximum principle $\mu$ must be constant.
 If $\mu$ is a constant, then $j=0$.
 Thus $\theta$ must satisfy $\operatorname{div}d\nabla\theta=0$ by \eqref{EETI}.
 By \eqref{EBT1} for $\theta$ and the strong maximum principle, $\theta$ must be a constant.

For Type I\hspace{-0.1em}I we first prove that $\check{\mu}/\theta$ is a constant.
 Since
\[
    \frac{\partial(\check{\mu}/\theta)}{\partial\nu}
    =\frac1\theta \frac{\partial\check{\mu}}{\partial\nu}
    -\frac{\check{\mu}}{\theta^2} \frac{\partial\theta}{\partial\nu}
    =0 \quad\text{on}\quad \partial\Omega
\]
by \eqref{EBT2} and since
\[
    \operatorname{div}\left(\tilde{\gamma}\nabla(\check{\mu}/\theta)\right)=0,
\]
the strong maximum principle yields $\check{\mu}/\theta$ is a constant so that $j=0$.
 This implies that $\theta$ is constant by $\operatorname{div}(d\nabla\theta)=0$ with $\partial\theta/\partial\nu=0$.
 Thus $\check{\mu}$ is a constant.
\end{proof}

\section{Linearized operator} \label{SL}

We shall analyze Type I of the Cahn-Hilliard equations.
 We start with the mass specific Helmholtz energy $\psi$ in the form of \eqref{EH1}.
 Namely,
\begin{equation} \label{EHEx}
    \psi(w,\theta,\nabla w)
    =\psi_0(w,\theta)+\frac{\varepsilon(w,\theta)}{2}|\nabla w|^2.
\end{equation}
We assume that $\varepsilon$ is independent of $\theta$ so that $\eta_{\nabla w}=0$.
 The Cahn-Hilliard equations \eqref{EETI} are now of the form
\begin{equation}
    \begin{split} \label{ECH1}
    & w_t =-\operatorname{div}j, \quad
    j=-\gamma\nabla\mu,\quad
    \mu=\delta_w\psi, \\
    & \kappa\theta_t =\operatorname{div}(d\nabla\theta)
    +\gamma|\nabla\mu|^2+\theta\eta_w\operatorname{div}j,
    \end{split}
\end{equation}
where $\kappa=-\theta\psi_{\theta\theta}$, $\eta=-\psi_\theta$.
 Since $\varepsilon$ in \eqref{EHEx} is assumed to be independent of $\theta$,
\begin{align*}
    & \psi_\theta=\psi_{0\theta},\quad
    \kappa=-\theta\psi_{0\theta\theta},\quad
    \eta=-\psi_{0\theta},\quad
    \eta_w=-\psi_{0\theta w}, \\
    & \delta_w\psi=\psi_{0w}+\frac{\varepsilon_w}{2}|\nabla w|^2-\operatorname{div}(\varepsilon\nabla w).
\end{align*}
As we see in Section~\ref{SD}, the second law of thermodynamics yields $\gamma\ge0$, $d\ge0$.
 Since $\varepsilon$ is introduced as regularizing the equation, we assume that $\varepsilon>0$.
 We also assume that the volume specific heat content  $\kappa=\epsilon_\theta$ is positive by a physical reason.
 To avoid technical difficulty and complexity, in the sequel we shall assume that
\begin{description}
    \item[\rm (CH1)] $\gamma$, $\kappa$, $d$ and $\varepsilon$ are positive constants.
\end{description}
Under this assumption, $\psi_w=\psi_{0w}$ and $\mu=\psi_{0w}-\varepsilon\Delta w$ in \eqref{ECH1}.

The equation \eqref{ECH1} is still quasi-linear.
 In fact, we observe that
\begin{align*}
    -\operatorname{div}j &= \gamma\Delta\mu
    =\gamma\Delta\left(\psi_{0w}(w,\theta)-\varepsilon\Delta w\right), \\
    \Delta\psi_{0w} &= \operatorname{div}(\psi_{0w\theta}\nabla\theta)
    +\operatorname{div}(\psi_{0ww}\nabla w)\\
    &= -\eta_w\Delta\theta+\psi_{0ww}\Delta w \\
    +&\sum_{j=1}^n\partial_j U^T \partial_{w,\theta}^2\psi_{0w}\partial_j U,\quad
    U=(w,\theta)^T,
\end{align*}
where $\partial_{w,\theta}^2 f$ denotes the Hesse matrix of $f$.
 Thus
\[
    \operatorname{div}j=\gamma\varepsilon\Delta^2 w
    +\gamma\eta_w\Delta\theta
    +\text{lower order terms}.
\]
The equation \eqref{ECH1} is formally rewritten in the form
\begin{equation} \label{EAb1}
    \frac{dU}{dt}+A(U)U=F(U),
\end{equation}
where
\begin{align*}
    A(U)&=
    \begin{pmatrix}
    \gamma\varepsilon\Delta^2 & \gamma\eta_w\Delta \\
    -\kappa^{-1}\theta\eta_w\gamma\varepsilon\Delta^2 
    & -\kappa^{-1}(d+\theta\eta_w^2\gamma)\Delta
    \end{pmatrix}, \quad
    \eta_w=-\psi_{0\theta w}(w,\theta) \\
    F(U)&=
    \dbinom{\gamma f(U)}{\kappa^{-1}\gamma|\nabla\mu|^2
    +\kappa^{-1}\theta\eta_w\gamma f(U)}, \quad
    U=(w,\theta)^T
\end{align*}
with
\[
    f(U)=\psi_{0ww}\Delta w+\sum_{j=1}^n \partial_jU^T
    \partial_{w,\theta}^2\psi_{0w}\partial_jU.
\]

\subsection{Linearization} \label{SSL}

To discuss the stability of an equilibrium, it is useful to linearize the equation \eqref{EAb1} around the equilibrium.
 Let $(w^e,\theta^e)$ be an equilibrium of \eqref{ECH1} (under boundary condition \eqref{EBT1} if $\Omega$ has a boundary).
 We set $w=w^e+\tilde{w}$, $\theta=\theta^e+\tilde{\theta}$.
 By Proposition~\ref{PST}, we see that
\[
    \mu^e=\psi_{0w}(w^e,\theta^e)-\varepsilon\Delta w^e
\]
is a constant.
 Since
\begin{align*}
    \mu &=\psi_{0w}(w^e+\tilde{w},\theta^e+\tilde{\theta})
    -\varepsilon\Delta(w^e+\tilde{w}) \\
    &=\psi_{0ww}(w^e,\theta^e)\tilde{w}+\psi_{0w\theta}(w^e,\theta^e)\tilde{\theta}
    +\psi_{0w}(w^e,\theta^e)-\varepsilon\Delta w^e-\varepsilon\Delta\tilde{w} \\
    &+o\left(|\tilde{w}|+|\tilde{\theta}|\right) \\
    &= -\eta_w(w^e,\theta^e)\tilde{\theta}
    +\psi_{0ww}(w^e,\theta^e)\tilde{w}
    -\varepsilon\Delta\tilde{w}
    +\mu^e+o\left(|\tilde{\theta}|+|\tilde{w}|\right)
\end{align*}
as $\tilde{\theta}\to0$, $\tilde{w}\to0$, $\nabla\mu^e=0$ implies that
\[
    \nabla\mu=\nabla(-\varepsilon\Delta\tilde{w}+\psi_{0ww}^e\tilde{w})
    -\nabla(\eta_w^e\tilde{\theta})
    +o\left(|\nabla\tilde{w}|+|\nabla\tilde{\theta}|+|\tilde{w}|+|\tilde{\theta}|\right),
\]
where $\eta_w^e=\eta_w(w^e,\theta^e)$, $\psi_{0ww}^e=\psi_{0ww}(w^e,\theta^e)$.
 The term $|\nabla\mu|^2$ is a term of degree more than or equal to two for $\tilde{\theta}$, and $\tilde{w}$ so it is ignored in this linearization.
 We also linearize the term $\theta\eta_w\operatorname{div}\nabla\mu$ at $(w^e,\theta^e)$.
 Since $\mu^e$ is a constant, it turns out that linear part of $\theta\eta_w\operatorname{div}\nabla\mu$ with respect to $\tilde{w}$, $\tilde{\theta}$ equals $\theta^e\eta_w^e$ times linear part of $\operatorname{div}\nabla\mu$.
 Thus, the linearized operator $A_0$ of $A(U)U-F(U)$ at $U=(w^e,\theta^e)^T$ equals
\[
    A_0\tilde{U}=
    \dbinom{\operatorname{div}J}{-\kappa^{-1}\theta^e\eta_w^e\operatorname{div}J
    -\kappa^{-1}d\Delta\theta}
\]
with $J=\gamma\nabla(\varepsilon\Delta\tilde{w}-\psi_{0ww}^e\tilde{w})+\gamma\Delta(\eta_w^e\tilde{\theta})$.
 In other words,
\[
    A_0\tilde{U}=
    \begin{pmatrix}
    \gamma\varepsilon\Delta^2-\gamma\Delta(\psi_{0ww}^e\cdot) & \gamma\Delta(\eta_w^e\cdot) \\
    -\kappa^{-1}\theta^e\eta_w^e\left(\gamma\varepsilon\Delta^2-\gamma\Delta(\psi_{0ww}^e\cdot)\right) 
    & -\kappa^{-1}\left(d\Delta+\gamma\theta^e\eta_w^e\Delta(\eta_w^e\cdot)\right)
    \end{pmatrix}
    \dbinom{\tilde{w}}{\tilde{\theta}}.
\]

We next fix function spaces to define $A$, $F$ and $A_0$.

\subsection{Function spaces} \label{SSF}

Let $q$ be $1
<q<\infty$.
 We take a function space
\[
    X_0=L^q(\Omega)\times L^q(\Omega).
\]
Let $L$ be the negative Laplace operator in $L^q(\Omega)$ (with the Neumann boundary condition if $\Omega$ has a boundary).
 Namely, if $\Omega$ is a smooth bounded domain in $\mathbb{R}^n$, then
\[
    Lz=-\Delta z, \quad
    z\in D(L)
\]
with the domain
\[
    D(L)=\left\{z\in W^{2,q}(\Omega) \biggm|
    \frac{\partial z}{\partial\nu}=0\ \text{on}\ \partial\Omega\right\}.
\]
In the case $\Omega=\mathbb{T}^n$, $D(L)$ is simply $W^{2,q}(\Omega)$, where $W^{m,q}$ denotes the $L^q$-Sobolev space of order $m=1,2,\ldots$.
 We shall rewrite \eqref{EAb1} under \eqref{EBT1}.
 By rescaling time by $\tau=t\gamma$, we may assume $\gamma=1$ by changing $d$ by $d/\gamma$.
 We set $X_1=D(L^2)\times D(L)$.
 Our system \eqref{ECH1} with \eqref{EBT1} for $U^T\in X_1$ is now written as \eqref{EAb1} by setting
\begin{equation} \label{EOpA}
    A(U)=
    \begin{pmatrix}
    \varepsilon L^2 & -\eta_w L \\
    -\kappa^{-1}\theta\eta_w\varepsilon L^2 
    & \kappa^{-1}(d+\theta\eta_w^2)L
    \end{pmatrix}
\end{equation}
with
\begin{equation} \label{EOpF1}
    F(U)=
    \dbinom{f(U)}{\kappa^{-1}|\nabla\mu|^2
    +\kappa^{-1}\theta\eta_w f(U)},
\end{equation}
where
\begin{equation} \label{EOpF2}
    f(U)=-\psi_{0ww}Lw
    +\sum_{j=1}^n \partial_j U^T\partial_{w,\theta}^2\psi_{0w}\partial_j U.
\end{equation}
Indeed, $w\in D(L)$ and $Lw\in D(L)$ implies that $\partial(Lw)/\partial\nu=0$ on $\partial\Omega$.
 Since $\mu=\psi_{0w}+\varepsilon Lw$, $\partial w/\partial\nu=0$ with $\partial(Lw)/\partial\nu=0$ and $\partial\theta/\partial\nu=0$ yields
\[
    \frac{\partial\mu}{\partial\nu}
    =\psi_{0ww}(w,\theta)\frac{\partial w}{\partial\nu}
    +\psi_{0w\theta}(w,\theta)\frac{\partial\theta}{\partial\nu}+\frac{\partial}{\partial\nu}(Lw)=0.
\]
Thus the boundary condition \eqref{EBT1} for $w$ and $\mu$ are realized in this formulation.
 The linearized operator around $(w^e,\theta^e)$ can be written
\begin{equation} \label{EOpA0}
    A_0\tilde{U}=
    \begin{pmatrix}
    LL^\varepsilon & -L(\eta_w^e\cdot) \\
    -\kappa^{-1}\theta^e\eta_w^e LL^\varepsilon 
    & \kappa^{-1}\left(dL+\theta^e\eta_w^e L(\eta_w^e\cdot)\right)
    \end{pmatrix}
    \dbinom{\tilde{w}}{\tilde{\theta}}
\end{equation}
with
\begin{equation} \label{EOpB}
    L^\varepsilon w=\varepsilon L+\psi_{0ww}^e.
\end{equation}

\section{Local well-posedness} \label{SW}

There seems to be many ways to construct a local solution of \eqref{EAb1} with \eqref{EOpA}, \eqref{EOpF1}.
 Here we shall apply an abstract result \cite[Theorem~5.1.1]{PS}.
 Let $X_0$ and $X_1$ be two complex Banach spaces such that $X_1$ is densely embedded in $X_0$.
 Let $(X_0,X_1)_{\theta,q}$ denote its real interpolation space \cite{L2}, \cite{PS}.
 We set
\[
    X_\gamma=(X_0,X_1)_{1-1/p,p}
\]
for a given $p\in(1,\infty)$.
 Let $V$ be an open set of $X_\gamma$.
 We consider
\begin{equation} \label{EAF}
    (\mathcal{A},\mathcal{F}):V\to\mathcal{L}(X_1,X_0)\times X_0
\end{equation}
and a quasilinear evolution equation
\begin{equation} \label{EAb2}
    \frac{dv}{dt}+\mathcal{A}(v)v=\mathcal{F}(v)
\end{equation}
with initial condition $v(0)=v_0\in V$.
 Here, $\mathcal{L}(X_1,X_0)$ is the space of all bounded linear operators from $X_1$ to $X_0$.
 Let $\mathbb{E}_0$ be a base space defined by
\[
    \mathbb{E}_0=L^p(0,T; X_0)
\]
and $\mathbb{E}_1$ be a solution space defined by
\[
    \mathbb{E}_1=W^{1,p}(0,T;X_0)\cap L^p(0,T; X_1)
\]
for $T>0$.
\begin{prop} \label{PPSL}
Assume that $p\in(1,\infty)$.
 Assume that $v_0\in V$ and that the mapping $(\mathcal{A},\mathcal{F})$ in \eqref{EAF} is Lipschitz continuous.
 Assume that $\mathcal{A}(v_0)$ has the ($L^p$-)maximal regularity on $(0,T)$, i.e.,
\[
    v\mapsto\left(\frac{dv}{dt}+\mathcal{A}(v_0)v,v(0)\right)
\]
is an isomorphism from $\mathbb{E}_1$ to $\mathbb{E}_0\times X_\gamma$.

Then there exist $T_0=T_0(v_0)>0$ and $\delta=\delta(v_0)$ such that if $v_1\in V$ satisfies $\|v_1-v_0\|_{X_\gamma}\le\delta$ then \eqref{EAb2} with $v(0)=v_1$ has a unique solution $v\in\mathbb{E}_1\cap C\left([0,T_0];V\right)$ on $[0,T_0]$.
 The mapping $v_1\mapsto v$ is locally Lipschitz continuous.
\end{prop}

In \cite[Theorem~5.1.1]{PS}, $T$ is assumed to be $\infty$.
 However, finite $T$ is enough to prove local existence.
 For function spaces, more general setting is included.
 Other than the space $X_\gamma$, the trace space of weighted Sobolev space $X_{\gamma,\mu}$ is discussed.
 Moreover, regularizing effect is also discussed.

\subsection{Abstract form} \label{SSAb}

We set $X_0=L^q(\Omega)\times L^q(\Omega)$ ($1<q<\infty$) and $X_1=D(L^2)\times D(L)$ so that
\[
    X_\gamma=(X_0,X_1)_{1-1/p,p}
    \subset B_{q,p}^{4(1-1/p)}(\Omega)
    \times B_{q,p}^{2(1-1/p)}(\Omega),
\]
where $B_{q,p}^s$ is an $L^q$-Besov space of order $s$ with summability exponent $p$.
 For $U=(w,\theta)^T\in X_\gamma$ we define $A(U)$ and $F(U)$ by \eqref{EOpA} and \eqref{EOpF1}, \eqref{EOpF2}, where
\[
    \eta_w(w,\theta)=-\psi_{0w\theta}(w,\theta)
\]
since we have assumed that $\psi=\psi_0+\varepsilon|\nabla w|^2/2$ with constant $\varepsilon>0$.
 By the Sobolev embedding (see e.g.\ \cite{Sa})
\begin{equation} \label{ESo}
    B_{q,p}^s(\Omega)\subset C(\bar{\Omega})\quad \text{for}\quad 
    \frac1q<\frac{s}{n},\quad s>0.
\end{equation}
For example, $\partial_j\in\mathcal{L}\left(B_{q,p}^{2(1-1/p)}(\Omega),C(\bar{\Omega})\right)$ if
\[
    \frac{n}{q}+\frac2p<1
\]
since $B_{q,p}^{1-2/p}(\Omega)\subset C(\bar{\Omega})$ for this choice of exponent.
\begin{prop} \label{PAR}
Assume that $1/p+n/(2q)<1$ for $p,q\in(1,\infty)$.
 Assume that $\psi_0$ is $C^3$ in its variables $(w,\theta)$.
 Then $A$ defined by \eqref{EOpA} is $C^1$ from $X_\gamma$ to $\mathcal{L}(X_1,X_0)$.
 Assume that $2/p+n/q<1$ and that $\psi_0$ is $C^4$.
 Then, $F$ defined by \eqref{EOpF1}, \eqref{EOpF2} is $C^1$ from $X_\gamma$ to $X_0$ for $p$ satisfying $4/p+n/(2q)<1$.
\end{prop}
\begin{proof}
By \eqref{ESo}, the mapping $(w,\theta)\mapsto\eta_w$ is $C^1$ from $X_\gamma$ to $C(\bar{\Omega})$ since $\psi_0$ is $C^3$.
 By definition of $A$, it is now clear that $A\in C^1\left(X_\gamma,\mathcal{L}(X_1,X_0)\right)$ since $B_{q,p}^{2(1-1/p)}(\Omega)\subset C(\bar{\Omega})$ if $1/p+n/(2q)<1$.

We next study $f(U)$ in \eqref{EOpF2}.
 Since
\[
    \left(L^q(\Omega),D(L^2)\right)_{1-1/p,p}\subset D(L)
    =D\left((L^2)^{1/2}\right)
    \quad\text{for}\quad p>2
\]
by comparing domains of fractional powers and real interpolation \cite{L2}, we observe that $(w,\theta)\mapsto Lw$ is a bounded linear operator from $X_\gamma$ to $L^q(\Omega)$.
 Since $(w,\theta)\mapsto\psi_{0ww}$ is $C^1$ from $X_\gamma$ to $C(\bar{\Omega})$, the first term of $f(U)$ is $C^1$ from $X_\gamma$ to $L^q(\Omega)$.
 By the Sobolev embedding \eqref{ESo} $\partial_j$ is a bounded linear operator from $X_\gamma$ to $C(\bar{\Omega})$.
 The exponent $(p,q)$ satisfying $2/p+n/q<1$ is taken so that $\partial_j U\in C(\bar{\Omega})\times C(\bar{\Omega})$ for $U^T=(w,\theta)\in X_\gamma$.
 Thus the second term of $f(U)$ is $C^1$ from $X_\gamma$ to $C(\bar{\Omega})$ since $\partial_{w,\theta}^2\psi_{0w}$ is $C^1$.
 We now conclude that $f$ is $C^1$ from $X_\gamma$ to $L^q(\Omega)$.

It remains to prove that $|\nabla\mu|^2$ is $C^1$ from $X_\gamma$ to $L^q(\Omega)$.
 We recall
\[
    \mu=\psi_{0w}+\varepsilon Lw
\]
so that
\[
    \nabla\mu=\psi_{0w\theta}\nabla\theta
    +\psi_{0ww}\nabla w+\varepsilon\nabla Lw.
\]
As we already observed, the first two terms are $C^1$ from $X_\gamma$ to $C(\bar{\Omega})$.
 Since $\nabla\Delta$ is bounded from $B_{q,p}^{4(1-1/p)}(\Omega)$ to $B_{q,p}^{1-4/p}(\Omega)$, the operator $\nabla Lw$ is a bounded linear operator from $\left(L^q(\Omega),D(L^2)\right)_{1-1/p,p}$ to $L^q(\Omega)$ provided that $4/p+n/(2q)<1$ since $B_{q,p}^{1-4/p}(\Omega)\subset L^{2q}(\Omega)$.
 We thus conclude that $\nabla\mu$ is $C^1$ from $X_\gamma$ to $\left[L^{2q}(\Omega)\right]^n$, and $|\nabla\mu|^2\in L^q(\Omega)$.
\end{proof}

\subsection{Maximal regularity} \label{SSM}

We shall prove $L^p$-maximal regularity on $(0,\infty)$ of $A(U^*)+\omega$($:= A(U^*)+\omega I$) for some $\omega>0$ by a kind of perturbation argument.
 (Here $I$ denotes the identity operator.
 We often suppress $I$ in the sequel.)
 To achieve this goal, we use notion of $\mathcal{R}$-boundedness of resolvent type operators as in \cite{PS}.
 Let us briefly recall the theory.

Let $X$ and $Y$ be Banach spaces and let $\mathcal{T}$ be a family of operators in $\mathcal{L}(X,Y)$.
 We say that $\mathcal{T}$ is \emph{$\mathcal{R}$-bounded} if there is a positive constant $C$ such that the inequality
\begin{equation} \label{ERB}
    \int_0^1\left\|\sum_{j=1}^k r_j(t)T_jx_j\right\|_Y^2 dt
    \le C\int_0^1\left\|\sum_{j=1}^k r_j(t)x_j\right\|_X^2 dt
\end{equation}
is valid for all $\{T_j\}_{j=1}^k\subset\mathcal{T}$ and $\{x_j\}_{j=1}^k\subset X$, where $r_j(t)$ is the Rademacher function defined by
\[
    r_j(t)=\operatorname{sgn}\left(\sin(2^j\pi t)\right).
\]
By Kahane's inequality, the estimate \eqref{ERB} is valid if the power $2$ is replaced by other power $p\in[1,\infty)$.
 The smallest $C$ in \eqref{ERB} is called \emph{$\mathcal{R}$-bound} for $\mathcal{T}$ and is denoted by $\mathcal{R}(\mathcal{T})$.
 See \cite[Definition~4.1.1]{PS} for more sophisticated definition.

Let $A$ be a closed operator with dense domain and dense range in $X$.
 Let
\[
    \mathcal{T}_\varphi(A)
    =\left\{ \lambda(\lambda+A)^{-1} \bigm|
    -\lambda\not\in\bar{\Sigma}_\varphi \right\}
\]
with $\Sigma_\varphi=\left\{z\in\mathbb{C}\bigm| |\arg z|<\varphi,\ z\neq0\right\}$.
 We say that $A$ is \emph{$\mathcal{R}$-sectorial} with angle $\varphi_A\in[0,\pi)$ if
\[
    \varphi_A
    =\inf\left\{ \varphi\bigm|
    \mathcal{R}\left(\mathcal{T}_\varphi(A)\right)<\infty \right\}
\]
is defined.
 This notion characterizes $L^p$-maximal regularity \cite[Theorem~4.4.4]{PS}.
\begin{lem} \label{LRMR}
Let $X$ be a Banach space of class $\mathcal{HT}$ and $1<p<\infty$.
 Then $\mathcal{A}$ has $L^p$-maximal regularity on $(0,\infty)$ if and only if $\mathcal{A}$ is $\mathcal{R}$-sectorial with angle $\varphi_{\mathcal{A}}<\pi/2$.    
\end{lem}
The class $\mathcal{HT}$ means that the Hilbert transform for $X$-valued function is always a bounded operator in $L^2(\mathbb{R},X)$.
 An important example of such an $X$ is $L^q$ space as well as Hilbert spaces.

We first note that the operators $L(\omega_0):=L+\omega_0$ and $L^2(\omega_0)=L^2+\omega_0$ have $L^p$-maximal regularity on $(0,\infty)$ for some $\omega_0>0$.
 This follows from a very general result for elliptic boundary problems studied in \cite[Chapter~6]{PS}.
 Since $-\Delta$ and $\Delta^2$ are strongly elliptic, a key point to check is that whether the boundary condition fulfills the Lopatinskii-Shapiro condition.
 For $-L$ it is easy to check.
 For $L^2$ it can be checked; see e.g.\ \cite{GK}.

By Lemma~\ref{LRMR}, $L(\omega_0)$ and $L^2(\omega_0)$ are $\mathcal{R}$-sectorial with angle less than $\pi/2$.
 In other words
\[
    r_\varphi\left(L(\omega_0)\right)
    :=\mathcal{R}\left(\mathcal{T}_\varphi\left(L(\omega_0)\right)\right)
    \quad\text{and}\quad
    r_\varphi\left(L^2(\omega_0)\right)
\]
are finite for some $\varphi<\pi/2$.
 We may assume that $L(\omega_0)$ and $L^2(\omega_0)$ have the bounded inverse in $L^q(\Omega)$ by taking $\omega_0$ larger.
 We note that
\[
    \mathcal{R}(\mathcal{T}+\mathcal{S})
    \le \mathcal{R}(\mathcal{T})+R(\mathcal{S})
\]
for $\mathcal{T}+\mathcal{S}=\{T+S\mid T\in\mathcal{T},\ S\in\mathcal{S}\}$ \cite[Proposition~4.1.6]{PS}.
 Since $L(\omega_0)\left(\lambda+L(\omega_0)\right)^{-1}=I-\lambda\left(\lambda+L(\omega_0)\right)^{-1}$, $r_\varphi\left(L(\omega_0)\right)<\infty$ implies
\begin{equation} \label{ERBL}
    \tilde{r}_\varphi\left(L(\omega_0)\right)
    :=\mathcal{R}\left(\tilde{\mathcal{T}}_\varphi\left(L(\omega_0)\right)\right)<\infty,
\end{equation}
where
\[
    \tilde{\mathcal{T}}_\varphi\left(L(\omega_0)\right)
    :=\left\{L(\omega_0)\left(\lambda+L(\omega_0)\right)^{-1} \bigm|
    -\lambda\not\in \bar{\Sigma}_\varphi\right\}.
\]
For later use, we observe that
\begin{equation} \label{EDiff}
    r_\varphi\left(dL(\omega_0)\right)
    =r_\varphi\left(L(\omega_0)\right),\quad
    \tilde{r}_\varphi\left(dL(\omega_0)\right)\le1
    +r_\varphi\left(dL(\omega_0)\right)
    =1+r_\varphi\left(L(\omega_0)\right)
\end{equation}
for $d>0$ since $dL(\omega_0)=dL+d\omega_0$.

Let us now consider our operator
\[
    A(U^*)=
    \begin{pmatrix}
    \varepsilon L & -\alpha \\
    -\varepsilon\alpha\beta L 
    & d'+\alpha^2\beta
    \end{pmatrix}L,
\]
where $\alpha=\eta_w^*=\eta_w(w^
*,\theta^*)$, $\beta=\theta^*/\kappa$, $d'=d/(\kappa\gamma)$ with $U^{*T}=(w^
*,\theta^*)$.
 We also set
\[
    A_{\omega_0}(U^*)=
    \begin{pmatrix}
    \varepsilon L^2(\omega_0) & -\alpha L(\omega_0) \\
    -\varepsilon\alpha\beta L^2(\omega_0) 
    & (d'+\alpha^2\beta)L(\omega_0)
    \end{pmatrix}.
\]
This can be rewritten as
\begin{equation} \label{EMS}
    A_{\omega_0}(U^*)=S
    \begin{pmatrix}
    \varepsilon L^2(\omega_0) & 0 \\
    0 & L(\omega_0)
    \end{pmatrix},\quad
    S=
    \begin{pmatrix}
    1 & -\alpha \\
    -\alpha\beta & d'+\alpha^2\beta
    \end{pmatrix}
\end{equation}
or
\begin{equation} \label{ERep}
    A_{\omega_0}(U^*)=A(U^*)+S
    \begin{pmatrix}
    \varepsilon & 0 \\
    0 & 1
    \end{pmatrix}\omega_0.
\end{equation}
This $S$ does not contain any differential operators.
 It is easy to see that $\det S=d'$ so that
\[
    S^{-1}=\frac{1}{d'}
    \begin{pmatrix}
    d'+\alpha^2\beta & \alpha \\
    \alpha\beta & 1
    \end{pmatrix}.
\]
\begin{thm} \label{TMR}
Assume that $1/p+n/(2q)<1$ for $p,q\in(1,\infty)$.
 Assume that $\varphi\in C^3$ so that $\alpha\in C(\bar{\Omega})$ for $U^{*T}=(w^*,\theta^*)\in X_\gamma$ and $\theta^*>0$.
 Then there exists a constant $\delta_0>0$ depending only on $L$ such that if $(\eta_w^*)^2\theta^*\gamma/d\le\delta_0$, then $A_{\omega_0}(U^*)$ is $\mathcal{R}$-sectorial with angle less than $\pi/2$ in $X_0$.
\end{thm}
\begin{proof}
Since $L(\omega_0)$ and $L^2(\omega_0)$ have bounded inverse in $L^q(\Omega)$ and since $S^{-1}$ in \eqref{EMS} is bounded in $X_0$, we observe that $R\left(A_{\omega_0}(U^*)\right)=X_0$.
 The density of $D\left(A_{\omega_0}(U^*)\right)$ in $X_0$ is clear since the domains of $L$ and $L^2$ are dense in $L^q(\Omega)$.

We shall use idea of decomposition of block operator matrices in \cite{AH}.
 We essential check the assumption for $M_1^{-1}(\lambda)$ of \cite[Proposition~4.8]{AH}.
 However, it is more instructive to analyze our operator directly from their point of view so we give a full proof.
 We consider an operator matrix
\[
    \mathcal{A}=
    \begin{pmatrix}
    A & B \\
    C & D
    \end{pmatrix}, \quad
    \mathcal{A}_d=
    \begin{pmatrix}
    A & 0 \\
    0 & D
    \end{pmatrix}.
\]
Formally
\begin{equation} \label{EADe}
    \lambda+\mathcal{A}
    =M(\lambda)(\lambda+\mathcal{A}_d)
\end{equation}
with
\[
    M(\lambda)=
    \begin{pmatrix}
    I & B(\lambda+D)^{-1} \\
    C(\lambda+A)^{-1} & I
    \end{pmatrix},
\]
where $I$ denotes the identity operator.
 We further decompose
\begin{equation} \label{EDec}
    M(\lambda)=
    \begin{pmatrix}
    I & B(\lambda+D)^{-1} \\
    0 & I
    \end{pmatrix}
    \begin{pmatrix}
    M_1(\lambda) & 0 \\
    C(\lambda+A)^{-1} & I
    \end{pmatrix}
\end{equation}
with
\[
    M_1(\lambda)
    =I-B(\lambda+D)^{-1}C(\lambda+A)^{-1}.
\]

We now consider $\lambda+A_{\omega_0}(U^*)$.
 Since
\begin{align*}
    \lambda+ A_{\omega_0}(U^*)&=
    \begin{pmatrix}
    \lambda+\varepsilon L^2(\omega_0) & -\alpha L(\omega_0) \\
    -\alpha\beta\varepsilon L^2(\omega_0) 
    & \lambda+(d'+\alpha^2\beta)L(\omega_0)
    \end{pmatrix} \\
    &= \begin{pmatrix}
    1 & 0 \\
    -\alpha\beta 
    & 1
    \end{pmatrix}
    \begin{pmatrix}
    \lambda+\varepsilon L^2(\omega_0) & -\alpha L(\omega_0) \\
    \lambda\alpha\beta 
    & \lambda+d'L(\omega_0)
    \end{pmatrix}
\end{align*}
so that
\[
    \left(\lambda+A_{\omega_0}(U^*)\right)^{-1}
    =(\lambda+\mathcal{A})^{-1}
    \begin{pmatrix}
    1 & 0 \\
    \alpha\beta & 1
    \end{pmatrix}, \quad
    \mathcal{A}=\begin{pmatrix}
    \varepsilon L^2(\omega_0) & -\alpha L(\omega_0) \\
    \lambda\alpha\beta & d'L(\omega_0),
    \end{pmatrix}
\]
where we invoke $\begin{psmallmatrix}
1 & 0\\-\alpha\beta & 1\end{psmallmatrix}^{-1}=\begin{psmallmatrix}1 & 0\\\alpha\beta & 1\end{psmallmatrix}$.
 Since
\[
    \mathcal{R}(\mathcal{TS})\le\mathcal{R}(\mathcal{T})\mathcal{R}(\mathcal{S})
    \quad\text{for}\quad
    \mathcal{TS}=\left\{ TS \mid T\in\mathcal{T},\ S\in\mathcal{S}\right\}
\]
by \cite[Proposition~4.1.6]{PS} and multiplication by $\begin{psmallmatrix}1 & 0\\ \alpha\beta & 1\end{psmallmatrix}$ is bounded in $X_0$, we observe
\[
    \mathcal{R}\left(\mathcal{T}_\varphi\left(A_{\omega_0}(U^*)\right)\right)
    \le C\mathcal{R}\left(\mathcal{T}_\varphi(\mathcal{A})\right)
\]
with $C$ depending only on $\alpha\beta$.
 (Here, we take operator norm of matrices acting on $\mathbb{R}^2$.)
 By \eqref{EADe} we observe that
\[
    (\lambda+\mathcal{A})^{-1}=(\lambda+\mathcal{A}_d)^{-1}M(\lambda)^{-1}
\]
with $A=\varepsilon L^2(\omega_0)$, $B=-\alpha L(\omega_0)$, $C=\lambda\alpha\beta$, $D=d'L(\omega_0)$.
 By \cite[Proposition~4.1.6]{PS} we see
\[
    \mathcal{R}\left(\mathcal{T}_\varphi(\mathcal{A})\right)
    \le \mathcal{R}
    \left(\mathcal{T}_\varphi(\mathcal{A}_d)\right)\mathcal{R}(\mathcal{M}_\varphi^{-1})
\]
with
\[
    \mathcal{M}_\varphi^{-1}
    =\left\{M(\lambda)^{-1}
    \bigm| -\lambda\not\in\bar{\Sigma}_\varphi \right\}.
\]
Since $r_\varphi\left(d'L(\omega_0)\right)=r_\varphi\left(L(\omega_0)\right)$ and $r_\varphi\left(\varepsilon L^2(\omega_0)\right)=r_\varphi\left(L^2(\omega_0)\right)$ are finite, we obtain that $\mathcal{R}\left(\mathcal{T}_\varphi(\mathcal{A}_d)\right)$ is finite.

It remains to prove that $\mathcal{R}(\mathcal{M}_\varphi^{-1})<\infty$.
 In the decomposition \eqref{EDec}, $\mathcal{R}$-boundedness of the family
\[
    B(\lambda+D)^{-1}=-\alpha L(\omega_0)\left(\lambda+d'L(\omega_0)\right)^{-1}, \quad
    -\lambda\not\in\bar{\Sigma}_\varphi
\]
follows from $\tilde{r}_\varphi(d'L)<\infty$ and boundedness of $\alpha$.
 Similarly, $\mathcal{R}$-boundedness of the family
\[
    C(\lambda+A)^{-1}=\lambda\alpha\beta \left(\lambda+\varepsilon L^2(\omega_0)\right)^{-1}, \quad
    -\lambda\not\in\bar{\Sigma}_\varphi
\]
follows from $\tilde{r}_\varphi\left(\varepsilon L^2(\omega_0)\right)<\infty$.
 From these observations, to estimate $\mathcal{R}(\mathcal{M}_\varphi^{-1})$ it suffices to prove $M_1(\lambda)^{-1}$ exists and
\[
    \mathcal{R}(\mathcal{M}_{1\varphi}^{-1})<\infty
    \quad\text{for}\quad
    \mathcal{M}_{1\varphi}^{-1}=\left\{M_1(\lambda)^{-1}
    \bigm| -\lambda\not\in\bar{\Sigma}_\varphi \right\}
\]
since \eqref{EDec} yields a decomposition of
\begin{align*}
    M(\lambda)^{-1}&=
    \begin{pmatrix}
    M_1(\lambda) & 0 \\
    C(\lambda+A)^{-1} & I
    \end{pmatrix}^{-1}
    \begin{pmatrix}
    I & -B(\lambda+D)^{-1} \\
    0 & I
    \end{pmatrix},\\
    \begin{pmatrix}
    M_1(\lambda) & 0 \\
    C(\lambda+A)^{-1} & I
    \end{pmatrix}^{-1}
    &=\begin{pmatrix}
    M_1(\lambda)^{-1} & 0 \\
    -C(\lambda+A)^{-1} M_1(\lambda)^{-1} & I
    \end{pmatrix}.
\end{align*}
By definition,
\[
    M_1(\lambda)=I+\alpha L(\omega_0)\left(\lambda+d'L(\omega_0)\right)^{-1}
    \lambda\alpha\beta
   \left(\lambda+\varepsilon L^2(\omega_0)\right)^{-1}
   =I+G(\lambda).
\]
By \eqref{EDiff} and \cite[Proposition~4.1.6]{PS} we see that
\begin{align*}
    \mathcal{R}(\mathcal{G}_\varphi)
    &\le \|\alpha\|_\infty\|\alpha\beta\|_\infty\tilde{r}_\varphi\left(d'L(\omega_0)\right)r_\varphi\left(\varepsilon L^2(\omega_0)\right)/d' \\
    &\le \|\alpha^2\beta\|_\infty\left(1+r_\varphi\left(d'L(\omega_0)\right)\right)
    r_\varphi\left(L^2(\omega_0)\right)/d', \quad
    \mathcal{G}_\varphi=\left\{G(\lambda)\Bigm|-\lambda\not\in\bar{\Sigma}_\varphi\right\}.
\end{align*}
If we take $\delta_0=1/\left(2\left(1+r_\varphi\left(L(\omega_0)\right)r_\varphi\left(L^2(\omega_0)\right)\right)\right)$,
then $\mathcal{R}(\mathcal{G}_\varphi)\le 1/2$ provided that $\|\alpha^2\beta\|_\infty/d'\le\delta_0$.
 The inverse $M_1(\lambda)^{-1}$ is now well-defined as the Neumann series
\[
    M_1(\lambda)^{-1}=\sum_{j=0}^\infty\left(-G(\lambda)\right)^j.
\]
Since $\mathcal{R}(\mathcal{G}_\varphi)\le1/2$, we obtain
\[
    \mathcal{R}(\mathcal{M}_{1\varphi}^{-1})
    \le 1+\sum_{j=1}^\infty\mathcal{R}(-\mathcal{G}_\varphi)^j
    \le 1+\sum_{j=1}^\infty\left(\frac12\right)^j
    =2<\infty.
\]
We thus conclude that $\mathcal{R}(\mathcal{M}_\varphi^{-1})<\infty$ if $|\alpha^2\beta|\le\delta_0d'$.
 Since $\alpha=\eta_w^*$, $\beta=\theta^*/\kappa$ and $d'=d/(\gamma\kappa)$, the last condition can be rewritten as
\[
    (\eta_w^*)^2\theta^*\gamma/d\le\delta_0.
\]
We then conclude that $A_{\omega_0}(U^*)$ is $\mathcal{R}$-sectorial with angle $\varphi<\pi/2$.
\end{proof}
\begin{cor} \label{CMR}
Under the same hypotheses of Theorem~\ref{TMR}, $A(U^*)+\omega$ is $\mathcal{R}$-sectorial with angle $\varphi<\pi/2$ for some $\omega>0$.
\end{cor}
\begin{proof}
By \eqref{ERep}, $A_{\omega_0}(U^*)$ and $A(U^*)$ are the same up to lower order terms.
 Since $A_{\omega_0}(U^*)$ is $\mathcal{R}$-sectorial with angle $\varphi<\pi/2$ by Theorem~\ref{TMR}, applying \cite[Proposition~4.4.3]{PS} yields the desired result.
\end{proof}
\begin{rem} \label{RSPer}
If one applies a general perturbation result, one has to assume that off diagonal parts are small.
 In this case, one should assume that both $\alpha$ and $\alpha\beta$ are small which is evidently stronger than the smallness of $\alpha^2\beta$ for fixed $\varepsilon$ and $d'$.
\end{rem}
Since our $X_0$ is of course of class $\mathcal{HT}$, Corollary~\ref{CMR} together with Lemma~\ref{LRMR} yields maximal regularity.
\begin{cor} \label{CMR1}
Under the same hypotheses of Theorem~\ref{TMR}, the operator $A(U^*)+\omega$ has $L^h$-maximal regularity ($1<h<\infty$) on $(0,\infty)$.
 In particular, $A(U^*)$ has $L^h$-maximal regularity on $(0,T)$ for every $T\in(0,\infty)$.
\end{cor}
\begin{proof}
It turns out that the last two statements are equivalent as far as $X_0$ is of class $\mathcal{HT}$ as stated for example in \cite[p.~77]{KW}.

We give here a direct proof that maximal regularity of $A(U^*)+\omega$ on $(0,\infty)$ implies the maximal regularity of $A(U^*)$ on $(0,T)$, $T<\infty$.
 Consider
\begin{equation} \label{EEOr}
    U_t+A(U^*)U=f,\quad
    \left.U\right|_{t=0}=u_0.
\end{equation}
The function $v=e^{-\omega t}U$ solves
\begin{equation} \label{EESh}
    v_t+\left(A(U^*)+\omega\right)v=e^{-\omega t}f,\quad
    \left.v\right|_{t=0}=u_0.
\end{equation}
Thus the solvability for \eqref{EEOr} on $(0,T)$ is reduced to the solvability of \eqref{EESh} on $(0,\infty)$ so the proof is now complete.
\end{proof}

By Proposition~\ref{PAR} and Corollary~\ref{CMR1}, we are able to apply Proposition~\ref{PPSL} to get a local existence theorem.
 Let us recall function spaces.
 The base space is $X_0=L^q(\Omega)\times L^q(\Omega)$ and $\mathbb{E}_0=L^p(0,T;X_0)$.
 The solution space is $X_1=D(L^2)\times D(L)$ and
\[
    \mathbb{E}_1=W^{1,p}(0,T;X_0)\cap L^p(0,T;X_1).
\]
The trace space is $X_\gamma=(X_0,X_1)_{1-1/p,p}$.
\begin{thm} \label{TLW}
Assume that $2/p+n/q<1$ and $4/p+n/(2q)<1$ for $p,q\in(1,\infty)$.
 Assume that $\psi_0$ is $C^4$ in its variables $(w,\theta)$.
 Assume that $\psi$ is given as \eqref{EHEx} with constant $\varepsilon>0$.
 Assume (CH1).
 Let $\Omega$ be a smooth bounded domain or $\Omega=\mathbb{T}^n$.
 For $(w_0,\theta_0)\in X_\gamma$ assume that
\begin{equation} \label{ESmall}
    \eta_w(w_0,\theta_0)^2 \theta_0\gamma/d\le\delta_0,
\end{equation}
where $\delta_0$ is a constant in Theorem~\ref{TMR}.
 Then there exists $T(w_0,\theta_0)>0$ and $\delta=\delta(w_0,\theta_0)>0$ such that \eqref{ECH1} (with \eqref{EBT1} if $\partial\Omega\neq\emptyset$) admits a unique solution $(w,\theta)\in\mathbb{E}_1\cap C\left ([0,T],X_\gamma\right)$ on $[0,T]$ with initial data $(w_1,\theta_1)\in X_\gamma$ satisfying $\left\|(w_1,\theta_1)-(w_0,\theta_0)\right\|_{X_\gamma}\le\delta$.
 The mapping $(w_1,\theta_1)\mapsto(w,\theta)\in\mathbb{E}_1$ is locally Lipschitz from $X_\gamma$ to $\mathbb{E}_1\cap C\left([0,T],X_\gamma\right)$.
\end{thm}
\begin{rem} \label{RProS}
\begin{enumerate}
    \item[(i)] (conservation of total fraction). 
    By \eqref{EBT1} and the first equation of \eqref{ECH1}
    \[
    \frac{d}{dt}\int_\Omega w\,dx=-\int_\Omega\operatorname{div}j\,dx
    =\int_{\partial\Omega}\gamma\nabla\mu\cdot\nu\,d\mathcal{H}^{n-1}=0.
    \]
    \item[(i\hspace{-0.1em}i)] (positivity of temperature). 
    Since $\theta$ solves
    \[
    \kappa\theta_t=\operatorname{div}(d\nabla\theta)+\gamma|\nabla\mu|^2+\theta h(x,t)
    \]
    with some bounded continuous function $h$ by the choice of our exponent $p,q$, the maximum principle (cf.\ \cite{PW}) implies that $\theta>0$ if $\theta_0>0$.
    \item[(i\hspace{-0.1em}i\hspace{-0.1em}i)] (conservation of total energy). 
    Since our solution $(w,\theta)$ satisfies \eqref{EEC1}, we see that
    \begin{align*}
    \frac{d}{dt} \int_\Omega \epsilon\,dx
    &=\int_\Omega\operatorname{div}j_e\,dx
    =\int_{\partial\Omega}j_e\cdot\nu\,d\mathcal{H}^{n-1}\\
    &=\int_{\partial\Omega}\left(d\nabla\theta+\gamma\mu\nabla\mu+\varepsilon(\operatorname{div}j)\nabla w\right)
    \cdot\nu\,d\mathcal{H}^{n-1}=0
    \end{align*}
    by \eqref{EBT1}.
     This property holds for any types provided that $j_e\cdot\nu=0$ on $\partial\Omega$.
\end{enumerate}
\end{rem}

\section{Stability of equilibrium} \label{SS}

We recall a generalized stability principle \cite[Theorem~5.3.1]{PS}.
 We consider \eqref{EAb2} with \eqref{EAF}.
 Let $v_*$ denote an equilibrium of \eqref{EAb2}, i.e.,
\[
    \mathcal{A}(v_*)v_*=\mathcal{F}(v_*), \quad
    v_*\in V\cap X_1.
\]
We recall the notion of (linearly) normal stability.
 Let $\mathcal{E}$ be the set of all equilibria in $V\cap X_1$.
 We consider the linearized operator of $\mathcal{A}(v)v-\mathcal{F}(v)$ at $v=v_*$.
 We assume that $(\mathcal{A},\mathcal{F})$ given by \eqref{EAF} is $C^1$.
 Then the linearized operator $\mathcal{A}_0$ can be explicitly defined by
\[
    \mathcal{A}_0w=\mathcal{A}(v_*)w
    +\left(\mathcal{A}'(v_*)w\right) v_*
    -\mathcal{F}'(v_*)w, \quad
    w\in X_1.
\]
\begin{definition} \label{DNS}
We say that an equilibrium $v_*\in\mathcal{E}$ is (linearly) normally stable if the following four properties hold.
\begin{description}
    \item[\it (S1)] Near $v_*$, the set $\mathcal{E}$ is $C^1$ manifold of $X_1$ of dimension $m=1,2,\ldots$.
    \item[\it (S2)] The tangent space for $\mathcal{E}$ at $v_*$ is isomorphic to $\ker\mathcal{A}_0$.
    \item[\it (S3)] Zero is a semi-simple eigenvalue of $\mathcal{A}_0$.
    \item[\it (S4)] There is $\omega>0$ such that $\operatorname{Re}z>\omega$ if $z$ is a non-zero spectrum of $\mathcal{A}_0$.
\end{description}
In the case $m=0$ in (S1), i.e., $v_*$ is isolated, we simply say $v_*$ is stable if (S4) holds and $0$ is not a spectrum of $\mathcal{A}_0$.
\end{definition}
We state a generalized stability principle \cite[Theorem~5.3.1]{PS} for \eqref{EAb2} in a slightly different form.
\begin{prop} \label{PGST}
Let $X$ be a Banach space of class $\mathcal{HT}$.
 Assume that $(\mathcal{A},\mathcal{F})$ in \eqref{EAF} is $C^1$.
 Let $v_*$ be an equilibrium of \eqref{EAb2}.
 For $p\in(1,\infty)$ assume that $\mathcal{A}_0+\omega$ has $L^p$ maximal regularity $(0,\infty)$ for some $\omega>0$.
 If $v_*$ is linearly normally stable, then $v_*$ is nonlinearly stable in the sense that there is $\delta>0$ such that if $\|v_0-v_*\|_{X_\gamma}<\delta$, then there exists a unique global-in-time solution $v$ of \eqref{EAb2} with $v(0)=v_0$ and moreover $v(t)$ converges to some $v_{**}\in\mathcal{E}$ as $t\to\infty$ in $X_\gamma$ at exponential rate, i.e., $\left\|v(t)-v_{**}\right\|_{X_\gamma}\le c_0e^{-\lambda_0t}$ with some positive constants $c_0$ and $\lambda_0$.
 (Note that $v_{**}$ may not be $v_*$.)
\end{prop}
Of course, if $v_*$ is moreover isolated, then it is nonlinearly stable and $\left\|v(t)-v_*\right\|_{X_\gamma}\le c_0 e^{-\lambda_0t}$.

In \cite[Theorem~5.3.1]{PS}, the maximal regularity for $\mathcal{A}(v_*)$ on $(0,\infty)$ is assumed instead of $\mathcal{A}_0+\omega$ without assuming $X$ is of class $\mathcal{HT}$.
 As a key step, the operator $\mathcal{A}_0$ is decomposed as $\mathcal{A}_0=\mathcal{A}_{0c}\oplus\mathcal{A}_{0s}$ associated with decomposition $X_0=X_c\oplus X_s$ where $X_c$ is the zero-eigenspace of $\mathcal{A}_0$.
 This is possible by (S3).
 Of course $\mathcal{A}_{0s}+\omega$ has the maximal regularity in $(0,\infty)$ in $X_s$ space.
 Since we consider on $X_s$, there is no spectrum of $\mathcal{A}_{0s}$ outside of $\Sigma_\varphi$.
 In particular $0$ is not a spectrum.
 In this case, $\mathcal{A}_{0s}-\omega_1$ for sufficiently small $\omega_1>0$ has the maximal regularity on $(0,\infty)$.
 For more details, see e.g.\ \cite[p.~77]{KW}.

\subsection{Properties of kernel of $A_0$} \label{SSKer}

We recall the operator $A_0$ defined by \eqref{EOpA0}
with $d$ replaced by $d/\gamma$.
 We set
\[
    d'=d/(\kappa\gamma), \quad
    \alpha=\eta_w^e=\eta_w(w^e,\theta^e), \quad
    \beta=\theta^e/\kappa.
\]
Moreover, we set
\[
    L^e=\varepsilon L+\psi_{0ww}^e.
\]
Then $A_0$ can be rewritten as
\[
    A_0=
    \begin{pmatrix}
    LL^e & -L(\alpha\cdot) \\
    -\alpha\beta LL^e & d'L+\alpha\beta L(\alpha\cdot)
    \end{pmatrix}
\]
or
\[
    A_0U=
    \binom{LL^ew-L(\alpha\theta)}{-\alpha\beta LL^ew+d'L\theta+\alpha\beta L(\alpha\theta)}
    \quad\text{for}\quad
    U^T=(w,\theta)\in X_1.
\]
We notice that
\begin{equation} \label{ERepA0}
    A_0=
    \begin{pmatrix}
    1 & 0 \\
    -\alpha\beta  & 1
    \end{pmatrix}
    L
    \begin{pmatrix}
    L^e & -\alpha \\
    0 & d'
    \end{pmatrix}.
\end{equation}
We know that $\beta$ is a positive constant while $\alpha$ is still real-valued but it may not be a constant; $\alpha$ may change the sign.
 We shall characterize the kernel (or the nullity) of $A_0$ defined by
\[
    \ker A_0=\left\{U^T\in X_1 \mid
    A_0 U=0 \right\}.
\]
\begin{lem} \label{LKer}
Assume that $\Omega$ is a bounded domain in $\mathbb{R}^n$.
 Let $-L$ be the Laplace operator with the Neumann boundary condition in $L^q(\Omega)$ ($1<q<\infty$).
 Then
\[
    \ker A_0= \left\{(w,c_1)\in X_1 \bigm|
    c_1\in\mathbb{C},\ 
    L^ew-\alpha c_1/d'=c_2\ \text{with some}\ c_2\in\mathbb{C}\right\},
\]
where $\alpha=\eta_w^e$.
 In particular, if $\alpha=\eta_w^e$ is a constant, then
\[
    \ker A_0=\left\{(w,c_1)\in X_1 \bigm|
    c_1\in\mathbb{C},\ 
    L^ew=c_2\ \text{with some}\ c_2\in\mathbb{C}\right\}.
\]
The same results are valid when $\Omega=\mathbb{T}^n$.
\end{lem}
\begin{proof}
By our representation formula \eqref{ERepA0}, $(w,\theta)\in\ker A_0$ is equivalent to saying that
\[
    L^ew-\alpha\theta\in\ker L, \quad
    d'\theta\in\ker L.
\]
Since $\ker L=\{c\in\mathbb{C}\}$, we obtain the desired result.
\end{proof}

We next study relation of $\ker A_0$ and $\ker B^e$ where $B^e$ is defined as
\[
    B^e=P_1L^e, \quad
    P_1w=w-\frac{1}{|\Omega|} \int_\Omega w\,dx
\]
as an operator in $P_1L^q(\Omega)$ which is the space of all average free $L^q$ functions.
 Lemma~\ref{LKer} immediately implies the following relation.
\begin{cor} \label{CKer}
Assume that $\eta_w^e$ is a constant function.
 Then
\[
    \ker A_0 =
    \left\{(w+c,c_1)\in X_1 \bigm|
    w\in\ker B^e\ \text{(in}\ P_1L^q\text{)},\ 
    c,c_1\in\mathbb{C} \right\}.
\]
\end{cor}

Let us consider structure of the set of equilibria.
 For a given $\theta^e\in\mathbb{R}$ and $w_{\mathrm{av}}\in\mathbb{R}$, we set
\[
    \mathcal{E}_{\theta^e}^{w_{\mathrm{av}}}
    =\left\{ w\in D(L^2) \biggm|
    \delta_w\psi(w,\theta^e)=\text{const.},\ 
    \frac{1}{|\Omega|}\int_\Omega w\,dx=w_{\mathrm{av}} \right\},
\]
where $w$ is considered a real-valued function.
 Let $\mathcal{E}$ denote the set of all (real-valued) equilibria of \eqref{ECH1} subject to \eqref{EBT1}.
 Then
\[
    \mathcal{E}=\left\{(w,\theta^e)\in X_1 \bigm|
    w\in\mathcal{E}_{\theta^e}^{w_{\mathrm{av}}},\
    w_{\mathrm{av}}\in\mathbb{R},\
    \theta^e\in\mathbb{R} \right\}.
\]
In general, at some point of $\mathcal{E}_{\theta^e}^{w_{\mathrm{av}}}$ the dependence of $\theta^e$ and $w_{\mathrm{av}}$ may not be smooth.
 For example, possibility of coexistence of two phases may depend on temperature.
\begin{prop} \label{PEq}
Assume that $\mathcal{E}_{\theta^e}^{w_{\mathrm{av}}}$ is $m$-dimensional manifold near $w^e$ for $\int_\Omega w^e\,dx=|\Omega|w_{\mathrm{av}}$.
\begin{enumerate}
    \item[(i)] If $w_{\mathrm{av}}$ and $\theta^e$ are smooth parameters of $\mathcal{E}_{\theta^e}^{w_{\mathrm{av}}}$, then $\mathcal{E}$ is $m+2$ dimensional manifold near $(w^e,\theta^e)$.
    \item[(i\hspace{-0.1em}i)] If $\eta_w(w,\theta)$ is a constant (which may depend on $w$, $\theta$), then $\mathcal{E}_{\theta^e}^{w_{\mathrm{av}}}$ is independent of $\theta^e$.
\end{enumerate}
\end{prop}
\begin{proof}
The first assertion is clear.
 If $\eta_w(w,\theta)$ is a constant $\eta_0$, then $\psi_{0w}(w,\theta)=\psi_{00}(w)-\theta\eta_0$ with some function $\psi_{00}$ depending only on $w$.
 The equation $L^ew=\mathrm{const}$ with $L^\varepsilon=\varepsilon L+\psi_{0ww}^e$ is unchanged even if we change $\theta$.
\end{proof}
\begin{cor} \label{CND}
Assume that $\mathcal{E}_{\theta^e}^{w_{\mathrm{av}}}$ is $m$-dimensional manifold near $w^e$ for $\int_\Omega w^e\,dx=|\Omega|w_{\mathrm{av}}$.
 Assume that $w_{\mathrm{av}}$ and $\theta^e$ are smooth parameters.
 If $w^e$ is non-degenerate, then the dimension of $\mathcal{E}$ near $(w^e,\theta^e)$ equals $\dim \ker A_0$ provided that $\eta_w^e$ is a constant function.
\end{cor}

This follows from Corollary~\ref{CKer} and Proposition~\ref{PEq}.

\subsection{Spectrum of $A_0$} \label{SSS}

We next study spectral properties of $A_0$.
 We first consider the case $q=2$ with special choice of $\beta$ and $d'$.
 In the sequel, we denote the complex conjugate of $w$ by $\bar{w}$.
\begin{lem} \label{LEG}
Let $\lambda\in\mathbb{C}$ be an eigenvalue of $A_0$ in $X_0=L^2(\Omega)\times L^2(\Omega)$.
 Assume that $d'\beta=1$ (i.e., $\theta^e=\kappa^2\gamma/d$) and $\int_\Omega\bar{w}L^ew\, dx\ge0$ for $w\in D(L)$.
 Then $\lambda$ is a non-negative real number.
\end{lem}
\begin{proof}
Since
\[
    \begin{pmatrix}
    1 & 0 \\
    \tau & 1
    \end{pmatrix}^{-1}
    =
    \begin{pmatrix}
    1 & 0 \\
    -\tau & 1
    \end{pmatrix}
\]
so that
\[
    A_0=
    \begin{pmatrix}
    1 & 0 \\
    -\alpha\beta & 1
    \end{pmatrix}
    A^e
    \begin{pmatrix}
    1 & 0 \\
    -\alpha\beta & 1
    \end{pmatrix}^{-1},\quad
    A^e= L
    \begin{pmatrix}
    L^e & -\alpha \\
    0 & d'
    \end{pmatrix}
    \begin{pmatrix}
    1 & 0 \\
    -\alpha\beta & 1
    \end{pmatrix},
\]
we see that the sets of eigenvalues of $A_0$ and $A^e$ are the same.
 Assume that 
\begin{equation} \label{EEig}
A^eU=\lambda U
\end{equation}
for some $\lambda\in\mathbb{C}\setminus\{0\}$.
 Then $U^T=(w,\theta)\in\operatorname
 {im}L\times\operatorname
 {im}L$ so that $\int_\Omega w\,dx=0$, $\int_\Omega\theta\,dx=0$.
 Here $\operatorname{im}L$ denotes the image (range $R(L)$) of $L$, i.e.,
\[
    \operatorname{im}L=\left\{ v\in X_0 \bigm| v=Lu\ \text{for some}\ u\in D(L)\right\}.
\]
 Thus the inverse $L^{-1}\bar{w}$, $L^{-1}\bar{\theta}$ is uniquely determined if we impose $\int_\Omega L^{-1}\bar{w}\,dx=0$, $\int_\Omega L^{-1}\bar{\theta}\,dx=0$.
 We multiply $(L^{-1}\bar{w},L^{-1}\bar{\theta})$ from the left of \eqref{EEig} and observe that
\begin{equation} \label{EEig2}
    (\bar{w},\bar{\theta})
    \begin{pmatrix}
    L^e & -\alpha \\
    0 & d'
    \end{pmatrix}
    \begin{pmatrix}
    1 & 0 \\
    -\alpha\beta & 1
    \end{pmatrix}
    \binom{w}{\theta}
    =\lambda(wL^{-1}\bar{w}+\theta L^{-1}\bar{\theta}).
\end{equation}
Since
\[
    \begin{pmatrix}
    L^e & -\alpha \\
    0 & d'
    \end{pmatrix}
    \begin{pmatrix}
    1 & 0 \\
    -\alpha\beta & 1
    \end{pmatrix}
    =\begin{pmatrix}
    L^e+\alpha^2\beta & -\alpha \\
    -d'\alpha\beta & d'
    \end{pmatrix},
\]
the left-hand side of \eqref{EEig2} becomes
\[
    \bar{w}(L^ew+\alpha^2\beta w)
    -\bar{w}\alpha\theta
    +\bar{\theta}(-d'\alpha\beta w+d'\theta).
\]
It can be written as
\begin{equation} \label{EQP}
    \bar{w}L^ew+Q(w_\alpha,\theta)
\end{equation}
if we set $w_\alpha=\alpha w$, where $Q$ is a quadratic (Hermite) form
\[
    Q(X,Y)=\beta |X|^2-\bar{X}Y-\beta d'X\bar{Y}
    +d'|Y|^2
\]
for $X,Y\in\mathbb{C}$.
 It is easy to see that real part $\operatorname{Re}Q\ge0$ for all $X,Y\in\mathbb{C}$ if and only if $\beta d'=1$.
 (If $\beta d'=1$, then
\[
    Q(w_\alpha,\theta)=\left| w_\alpha\sqrt{\beta}
    -\theta\sqrt{d'} \right|^2.\text{)}
\]
We integrate \eqref{EQP} to get
\[
    \int_\Omega \bar{w}L^ew\,dx
    +\int_\Omega Q(w_\alpha,\theta)\,dx
    =\lambda\int_\Omega(wL^{-1}\bar{w}+\theta L^{-1}\bar{\theta})\, dx
\]
with $\operatorname{Re}Q\ge0$.
 Since the integral of the right-hand side is positive for $(w,\theta)\neq(0,0)$ and $\int_\Omega\bar{w}L^ew\,dx\ge0$, we observe that $\lambda$ must be real and it must be non-negative.
\end{proof}

We continue to study eigenvalues for $q\neq2$, $1<q<\infty$.
\begin{lem} \label{LEGG}
Let $\lambda\in\mathbb{C}$ be a spectrum of $A_0$ in $X_0=L^q(\Omega)\times L^q(\Omega)$ for $q\in(1,\infty)$.
 Assume that $\theta^e=\kappa^2\gamma/d$ and that $\int_\Omega\bar{w}L^ew\,dx\ge0$ for $w\in D(L)$ with $q=2$.
 Then $\lambda$ is a non-negative eigenvalue.
 Moreover, the set of eigenvalues is discrete.
\end{lem}
\begin{proof}
We first note that $A_0$ is sectorial with angle less than $\pi/2$ by the maximal regularity.
 In particular, all spectra should be contained in $\Sigma_{\pi/2-\delta}\cup B_R$ with some $\delta>0$ and $R>0$, where $\Sigma_\theta=\left\{z\in\mathbb{C}\bigm||\arg z|<\theta,\ z\neq0\right\}$ and $B_R=\left\{z\in\mathbb{C}\bigm||z|<R\right\}$.
 Since $(\lambda-A_0)^{-1}$ is a compact operator in $X_0$ for $\lambda$ in the resolvent set, all spectra of $A_0$ must be eigenvalue and discrete by the Riesz- Schauder theory.
 Moreover, $(\lambda-A_0)^{-1}$ is a Fredholm operator of index zero.
 We may assume that $\lambda$ is an eigenvalue of $A_0$.
 Since $\bar{\Omega}$ is compact so that $L^q(\Omega)\supset L^2(\Omega)$ for $q\in(1,2]$, $\lambda$ must be a non-negative real number for $q\in(1,2]$ by Lemma~\ref{LEG}.

For $q\in(2,\infty)$, we consider a dual operator
\begin{align*}
    A_0^*
    &=\begin{pmatrix}
    L^eL & -L^eL(\alpha\beta\cdot) \\
    -\alpha L & d'L+\alpha L(\alpha\beta\cdot)
    \end{pmatrix} \\
    &=\begin{pmatrix}
    L^e & 0 \\
    -\alpha & d'
    \end{pmatrix}
    L\begin{pmatrix}
    1 & -\alpha\beta \\
    0 & 1
    \end{pmatrix} \\
    &=\begin{pmatrix}
    1 & -\alpha\beta \\
    0 & 1
    \end{pmatrix}^{-1}
    A^{*e}\begin{pmatrix}
    1 & -\alpha\beta \\
    0 & 1
    \end{pmatrix}, \quad
    A^{*e}=\begin{pmatrix}
    1 & -\alpha\beta \\
    0 & 1
    \end{pmatrix}
    \begin{pmatrix}
    L^e & 0 \\
    -\alpha & d'
    \end{pmatrix}L.
\end{align*}
As we argue in the proof of Lemma~\ref{LEG} for $A_0^e$, under our assumptions, any eigenvalue of $A_0^{*e}$ is real and non-negative in $L^2(\Omega)\times L^2(\Omega)$ space.
 Thus an eigenvalue $\lambda$ of $A_0^*$ in $L^q(\Omega)\times L^q(\Omega)$ for $1<q<2$ must be a non-negative real number.
 If $\lambda$ is an eigenvalue of $A_0$ for $q\in(2,\infty)$, then $\dim\operatorname{coker}(\lambda-A_0)\ge1$ since $\lambda-A_0$ is the Fredholm operator of index zero.
 This implies that $\bar{\lambda}$ is an eigenvalue of $A_0^*$ in $L^{q'}(\Omega)\times L^{q'}(\Omega)$ for $\frac1q+\frac{1}{q'}\equiv1$.
 Since $\bar{\lambda}$ must be real and non-negative, the proof is now complete.
\end{proof}

\subsection{Semi-simplicity} \label{SSSS}

We consider $A_0$ in $X_0=L^q(\Omega)\times L^q(\Omega)$.
\begin{lem} \label{LSS}
Assume that $\eta_w^e$ is a constant.
 Assume that zero is an eigenvalue of $A_0$.
 Then it is semi-simple.
 In other words, $\ker A_0\oplus\operatorname{im}A_0=X_0$.
\end{lem}
\begin{proof}
Since $A_0$ is a Fredholm operator of index zero, i.e., $\dim(\operatorname{coker}A_0)=\dim(\ker A_0)$.
 Thus, it suffices to prove that
\begin{gather*}
    \ker A_0\cap \operatorname{im} A_0=\{0\},\quad\text{i.e.,} \\
    A_0(A_0U)=0\quad\text{implies}\quad A_0U=0 \quad\text{for} \quad 
    U\in D(A_0^2).
\end{gather*}
Since $A_0(A_0U)=0$ implies $A_0U\in\ker A_0$, Corollary~\ref{CKer} implies that
\[
    A_0U\in\left\{(v+c,c_1)^T \bigm|
    c,c_1\in\mathbb{C},\ 
    v\in\ker B^e\ \text{in}\ L_{\mathrm{av}}^q(\Omega) \right\}.
\]
In other words, $U=(w,\theta)^T$ satisfies
\begin{gather}
    LL^ew- L(\alpha\theta)=v+c, \label{ESS1}\\
    -\alpha\beta\left(LL^ew- L(\alpha\theta)\right)+d'L\theta=c_1. \label{ESS2}
\end{gather}
Since $\alpha=\eta_w^e$ is assumed to be constant, $c_1\in D(L)$ as well as $v+c\in D(L)$.
 Thus $c=c_1=0$.
 Plugging \eqref{ESS1} to \eqref{ESS2}, we obtain
\[
    -\alpha\beta v+d'L\theta=0.
\]
Using this identity to erase $\theta$ in \eqref{ESS1}, we get
\begin{equation} \label{ESS3}
    LL^ew-\frac{\alpha^2\beta}{d'}v=v.
\end{equation}
Since $v\in\operatorname{im}L$, there is a unique $u\in D(L)$ with $\int_\Omega u\,dx=0$ such that $v=Lu$.
 The identity \eqref{ESS3} now implies that
\begin{equation} \label{ESS4}
    L^ew=\left(1+\frac{\alpha^2\beta}{d'}\right)u
\end{equation}
by applying $L^{-1}$ to both sides of \eqref{ESS3}.
 Multiplying $\bar{v}$ to \eqref{ESS4} and integrating over $\Omega$ we obtain
\begin{equation} \label{ESS5}
    \int_\Omega \bar{v}L^ew\,dx
    =\left(1+\frac{\alpha^2\beta}{d'}\right)
    \int_\Omega \bar{v}u\,dx.
\end{equation}
Integrating by parts, we see that
\[
    \int_\Omega \bar{v}L^ew\,dx
    =\int_\Omega (L^e\bar{v})w\,dx
    =\int_\Omega (\overline{L^ev})w\,dx=0
\]
since $v\in\ker B^e$.
 The identity \eqref{ESS5} now yields
\[
    0=\int_\Omega \bar{v}u\,dx
    =\int_\Omega u\overline{Lu}\,dx
    =\int_\Omega |\nabla u|^2\,dx.
\]
(This calculation is justified for $q\ge2$.
 Moreover, from a standard elliptic regularity theory, it follows that $\ker B^e$ is independent of $q\in(1,\infty)$ and $w\in W^{2,q}(\Omega)$ for any $q\in(1,\infty)$.
 Thus the above calculation is justified for all $q\in(1,\infty)$.)
 This implies $\nabla u=0$ so that $u$ is a constant.
 Since $\int_\Omega u\,dx=0$, we conclude that $u=0$ which implies $v=0$.
 In the system \eqref{ESS1}--\eqref{ESS2}, we conclude that $v=0$ with $c_1=c=0$.
 This says that $U^T=(w,\theta)\in\ker A_0$ and this is exactly what we would like to prove.
 The proof is now complete.
\end{proof}

\subsection{Application of generalized stability principle} \label{SSAP}

Let $\Omega$ be a bounded domain in $\mathbb{R}^n$ or $\Omega=\mathbb{T}^n$.
 We consider \eqref{ECH1} (with \eqref{EBT1} if $\partial\Omega\neq\emptyset$).
 We always assume (CH1), i.e., $\gamma$, $\kappa$, $d$ and $\varepsilon$ are positive constants.
 Let $(w^e,\theta^e)$ be an equilibrium so that $\theta^e$ is a constant, i.e., $\theta^e\in\mathbb{R}$.
 Let us state our main stability result.
 We set
\[
    w_{\mathrm{av}}:=\frac{1}{|\Omega|} \int_\Omega w^e\,dx
\]
and assume that the equilibrium set $\mathcal{E}_{\theta^e}^{w_{\mathrm{av}}}$ is an $m$-dimensional manifold near $(w^e,\theta^e)$ and it depends on $\theta$ at least $C^1$ near $\theta^e$.
 We recall function spaces $X_0=L^q(\Omega)\times L^q(\Omega)$, $X_1=D(L^2)\times D(L)$ and related spaces $\mathbb{E}_0$, $\mathbb{E}_1$ and $X_\gamma$.
\begin{thm} \label{TMS}
Assume that $\psi_0$ is $C^4$ and (CH1).
 Let $(w^e,\theta^e)\in C^3(\bar{\Omega})\times \mathbb{R}$ be an equilibrium.
 Assume that $w^e$ is an energetically stable equilibrium and it is non-degenerate at $\theta=\theta^e$.
 Assume further that $\eta_w^e=-\psi_{0\theta w}(w^e,\theta^e)$ is spatially constant and $\kappa$, $\gamma$, $d$ satisfy $\kappa^2\gamma=\theta^e d$. 
 Assume \eqref{ESmall} for $(w^e,\theta^e)$, i.e.,
\[
    \eta_w(w^e,\theta^e)^2 \theta^e\gamma/d\le\delta_0.
\]
For the exponent assume that $p,q\in(1,\infty)$ satisfy $2/p+n/q<1$ and $4/p+n/(2q)<1$.
 Then, there exists (small) $\sigma>0$ such that if the initial data $w_0=w|_{t=0}$, $\theta_0=\theta|_{t=0}$ with $(w_0,\theta_0)\in X_\gamma$ satisfies
\[
    \left\|(w_0-w^e,\theta_0-\theta^e)\right\|_{X_\gamma}\le\sigma, \quad
    \frac{1}{|\Omega|}\int_\Omega w_0\,dx=w_{\mathrm{av}}
    :=\frac{1}{|\Omega|}\int_\Omega w^e\,dx,
\]
then there exists a unique global-in-time solution $(w,\theta)\in\mathbb{E}_1\cap C\left([0,\infty),X_\gamma\right)$ of \eqref{ECH1} (with \eqref{EBT1} if $\partial\Omega\neq\emptyset$) with initial data $(w_0,\theta_0)$.
 The spatial average satisfies
\[
    \frac{1}{|\Omega|}\int_\Omega w\,dx = w_{\mathrm{av}}.
\]
Moreover there exists $(w_*,\theta_*)\in\mathcal{E}$ (near $(w^e,\theta^e)$) satisfies $\int_\Omega w_*\,dx=|\Omega|w_{\mathrm{av}}$ such that
\[
    \left\|\left(w(t)-w_*,\theta(t)-\theta_*\right)\right\|_{X_\gamma}\le c_0 e^{-\lambda_0t}
\]
with some $c_0>0$ and $\lambda_0>0$.
\end{thm}

We shall prove this result by applying Proposition~\ref{PGST} (\cite[Theorem~5.3.1]{PS}).
 For this purpose we write an equation \eqref{EAb1}, which is an abstract form of \eqref{ECH1} (with \eqref{EBT1} if $\partial\Omega\neq\emptyset$).
 Let $U^e=(w^e,\theta^e)^T$ and set $\tilde{U}=U-U^e$.
 Since
\begin{align*}
    A(U)U-A(U^e)U^e
    &=A(U)(U-U^e)+\left(A(U)-A(U^e)\right)U^e \\
    &=A(\tilde{U}+U^e)\tilde{U}+\left(A(\tilde{U}+U^e)-A(U^e)\right)U^e,
\end{align*}
the abstract equation \eqref{EAb1} can be written as
\begin{equation}
\begin{aligned} \label{EEAb}
    &\frac{d\tilde{U}}{dt}+\tilde{A}(\tilde{U})\tilde{U}=\tilde{F}(\tilde{U}), \quad
    \tilde{A}(\tilde{U})=A(\tilde{U}+U^e), \\
    &\tilde{F}(\tilde{U})=F(\tilde{U}+U^e)-F(U^e)
    -\left(A(\tilde{U}+U^e)-A(U^e)\right)U^e.
\end{aligned}
\end{equation}
We consider \eqref{EEAb}.
 We notice that the linearized operator near $\tilde{U}=0$ is $A_0$.
 We also note that $\int_\Omega w\,dx$ must be conserved as time develops.
 Thus it is more convenient to work in the space
\[
    X_{0\mathrm{av}}=L_{\mathrm{av}}^q(\Omega)\times L^q(\Omega), \quad
    L_{\mathrm{av}}^q(\Omega)=\left\{w\in L^q(\Omega) \biggm|
    \int_\Omega w\,dx=0 \right\}
\]
for $\tilde{U}$.
 We define $X_{1\mathrm{av}}=X_1\cap X_{0\mathrm{av}}$ and consider \eqref{EEAb} in
\[
    \mathbb{E}_{1\mathrm{av}}=W^{1,p}(0,T;X_{0\mathrm{av}})
    \cap L^p(0,T;X_{1\mathrm{av}})
\]
for initial data in $X_{\gamma\mathrm{av}}=X_\gamma\cap X_{0\mathrm{av}}$.
 However, we have to notice that $\tilde{A}$ may not preserve the average free condition if we take $A$ as in \eqref{EOpA}.
 We have to take
\[
    A(U)
    =\begin{pmatrix}
    \varepsilon L^2 & -L(\eta_w\cdot) \\
    -\kappa^{-1}\theta\eta_w\varepsilon L^2 & \kappa^{-1}\left(dL+\theta\eta_w L(\eta_w \cdot)\right)
    \end{pmatrix}
\]
and change $F(U)$ accordingly.
 Then $\tilde{A}(\tilde{U})=A(\tilde{U}+U^e)$ preserves the average free condition for $w$ since the image of $L$ is always average-free.
 With these preparations, we apply Proposition~\ref{PGST}.
\begin{proof}[Proof of Theorem~\ref{TMS}]
We set $(\mathcal{A},\mathcal{F})=(\tilde{A},\tilde{F})$ and $v_*=0$.
 By the change of definition, $\tilde{A}$ preserves the average free condition for the first component.
 The $C^1$ property of $(\tilde{A},\tilde{F})$ can be proved as in Proposition~\ref{PAR}.
 Our space $X_{0\mathrm{av}}$ is still of class $\mathcal{HT}$.
 By Corollary~\ref{CMR}, $\mathcal{A}_0+\omega$ for $\mathcal{A}_0=A_0$ is $\mathcal{R}$-sectorial of angle less than $\pi/2$ for some $\omega>0$ since our original $A$ and the linearized operator $A_0$ is the same up to lower order terms \cite[Proposition~4.4.3]{PS}.

Let us discuss spectral properties of $\mathcal{A}_0=A_0$, which is the linearized operator.
 By our assumption
\[
    \mathcal{E}^{w_{\mathrm{av}}}= \left\{(w,\theta)\in\mathcal{E} \bigm|
    w\in\mathcal{E}_\theta^{w_{\mathrm{av}}} \right\}
\]
is $m=k+1$ dimensional manifold near $(w^e,\theta^e)$ if $\mathcal{E}_\theta^{w_{\mathrm{av}}}$ is a $k$-dimensional manifold which may be an isolated point.
 The property (S1) holds.

If we assume that $\eta_w^e$ is a constant, (S2) follows from Corollary~\ref{CND}.

By the energetic stability of $w^e$, we conclude that any spectrum of $A_0$ must be a non-negative real number and discrete according to Lemma~\ref{LEGG}.
 Thus (S4) is proved.
 Here we invoke the assumption $\theta^e=\kappa^2\gamma/d$.

The property (S3) follows from Lemma~\ref{LSS}.
 We now apply Proposition~\ref{PGST} to get the desired result.
\end{proof}

We conclude this section by considering a special case when $w^e$ is a constant as discussed in \cite[Introduction]{GKTs}.
 In this case it is easy to see that $w_{\mathrm{av}}=w^e$ and if $-\psi_{0ww}(w^e,\theta^e)<\varepsilon\lambda_{\min}$, then $w^e$ is an energetically stable non-degenerate equilibrium.
 Here $\lambda_{\min}$ denotes the minimal positive eigenvalue of the Neumann Laplacian $L$.
 Moreover, $\mathcal{E}_{\theta^e}^{w_{\mathrm{av}}}=\{w_{\mathrm{av}}\}$ and $\mathcal{E}^{w_{\mathrm{av}}}$ near $\theta=\theta^e$ equals $\left\{(w_{\mathrm{av}},\theta)\bigm|\theta\in\mathbb{R}\right\}$.
\begin{cor} \label{CMS}
Assume the same hypotheses of Theorem~\ref{TMS}. 
 Assume that $-\psi_{0ww}(w^e,\theta^e)<\varepsilon\lambda_{\min}$ so that $w^e$ is a constant.
 Then the limit $(w_*,\theta_*)$ must satisfy $w_*=w^e=w_{\mathrm{av}}$.
\end{cor}

\section*{Acknowledgements}
This work was done as a part of research activities of Social Cooperation Program ``Mathematical Science for Refrigerant Thermal Fluids" sponsored by Daikin Industries, Ltd.\ at the University of Tokyo.
 The authors are grateful to members of the Technology and Innovation Center of Daikin Industries, Ltd.\ for showing several interesting phenomena related to phase transition with fruitful discussions which triggered this work.
 The work of the second author was partly supported by Japan Society for the Promotion of Science (JSPS) through the grants KAKENHI: JP24K00531, JP24H00183 and Daikin Industries, Ltd.\ through collaborative grants.
 The work of the fourth author was partly supported by JSPS through the grant JP22K13946.

\section*{Conflicts of Interest}
The authors declare no conflicts of interest.

\section*{Data Availability Statement}
No datasets were generated or analyzed during the current study.


\begin{thebibliography}{XXXX}

\bibitem[AGG]{AGG}
H.~Abels, H.~Garcke and G.~Gr\"un, 
Thermodynamically consistent, frame indifferent diffuse interface models for incompressible two-phase flows with different densities. 
{\it Math.\ Models Methods Appl.\ Sci.}\ {\bf22} (2012), no.~3, 1150013.

\bibitem[AH]{AH}
A.~Agresti and A.~Hussein, 
Maximal $L^p$-regularity and $H^\infty$-calculus for block operator matrices and applications. 
{\it J.\ Funct.\ Anal.}\ {\bf285} (2023), no.~11, Paper No.~110146.

\bibitem[AP]{AP}
H.~W.~Alt and I.~Pawlow, 
A mathematical model of dynamics of non-isothermal phase separation. 
{\it Physica D: Nonlinear Phenomena} {\bf59} (1992), 389--416.

\bibitem[BD]{BD}
D.~Bothe and W.~Dreyer, 
Continuum thermodynamics of chemically reacting fluid mixtures. 
{\it Acta Mech.}\ {\bf226} (2015), no.~6, 1757--1805.

\bibitem[BS]{BS}
D.~Bothe and K.~Soga, 
Thermodynamically consistent modeling for dissolution/growth of bubbles in an incompressible solvent. 
{\it Recent developments of mathematical fluid mechanics,} 111--134. Adv.\ Math.\ Fluid Mech. 
{\it Birkh\"auser/Springer, Basel,} 2016.

\bibitem[Ca]{Ca}
G.~Caginalp, 
The dynamics of a conserved phase field system: Stefan-like, Hele-Shaw, and Cahn-Hilliard models as asymptotic limits.
{\it IMA J.\ Appl.\ Math.}\ {\bf44} (1990), no.~1, 77--94.

\bibitem[C]{C}
J.~W.~Cahn, 
Free energy of a nonuniform system. I\hspace{-0.1em}I. Thermodynamic basis. 
{\it J.\ Chem.\ Phys.}\ {\bf30} (1959), 1121--1124.

\bibitem[CH1]{CH1}
J.~W.~Cahn and J.~E.~Hilliard, 
Free energy of a nonuniform system. I. Interfacial free energy. 
{\it J.\ Chem.\ Phys.}\ {\bf28} (1958), 258--267.

\bibitem[CH2]{CH2}
J.~W.~Cahn and J.~E.~Hilliard, 
Free energy of a nonuniform system. I\hspace{-0.1em}I\hspace{-0.1em}I. Nucleation in a two-component incompressible fluid.
{\it J.\ Chem.\ Phys.}\ {\bf31} (1959), 688--699.

\bibitem[CGR]{CGR}
M.~Conti, A.~Giorgini and G.~Ricchi, 
Global weak solutions to the Cahn-Hilliard equation with degenerate mobility and singular diffusion. 
arXiv: 2608.06371

\bibitem[DLSS]{DLSS}
F.~De Anna, C.~Liu, A.~Schl\"omerkemper and J.-E.~Sulzbach, 
Temperature dependent extensions of the Cahn-Hilliard equation. {\it Nonlinear Anal.\ Real World Appl.}\ {\bf77} (2024), Paper No.~104056.

\bibitem[DK]{DK}
R.~Denk and M.~Kaip, 
General parabolic mixed order systems in $L_p$ and applications. 
Oper.\ Theory Adv.\ Appl., 239. 
{\it Birkh\"auser/Springer, Cham,} 2013. 

\bibitem[FK1]{FK1}
H.~Freist\"uhler and M.~Kotschote, 
Phase-field and Korteweg-type models for the time-dependent flow of compressible two-phase fluids. 
{\it Arch.\ Ration.\ Mech.\ Anal.}\ {\bf224} (2017), 1--20.

\bibitem[FK2]{FK2}
H.~Freist\"uhler and M.~Kotschote, 
Phase-field descriptions of two-phase compressible fluid flow: Interstitial working and a reduction to Korteweg theory. 
{\it Quart.\ Appl.\ Math.}\ {\bf77} (2019), 489--496.

\bibitem[GKL]{GKL}
M.-H.~Giga, A.~Kirshtein and C.~Liu, 
Variational modeling and complex fluids. 
{\it Handbook of mathematical analysis in mechanics of viscous fluids,} 73--113.
{\it Springer, Cham,} 2018.

\bibitem[GKTs]{GKTs}
Y.~Giga, N.~Kajiwara and K.~Tsuda, 
On dynamic stability of energetically stable equilibria of the Navier-Stokes-Korteweg flows. 
preprint. arXiv:2608.05654

\bibitem[GK]{GK}
Y.~Giga and S.~Katayama, 
Asymptotically self-similar graph-like solutions to a multi-dimensional surface diffusion flow equation under contact angle and no-flux boundary conditions. 
{\it Discrete Contin.\ Dyn.\ Syst.}\ {\bf56} (2026), 285--314.

\bibitem[HM]{HM}
M.~Heida and J.~M\'alek, 
On compressible Korteweg fluid-like materials. 
{\it Internat.\ J.\ Engrg.\ Sci.}\ {\bf48} (2010), 1313--1324.

\bibitem[KW]{KW}
P.~C.~Kunstmann and L.~Weis, 
Maximal $L_p$-regularity for parabolic equations, Fourier multiplier theorems and $H^\infty$-functional calculus. {\it Functional analytic methods for evolution equations,} 65--311. 
Lecture Notes in Math., 1855.
{\it Springer-Verlag, Berlin,} 2004.

\bibitem[LT]{LT}
J.~Lowengrub and L.~Truskinovsky, 
Quasi-incompressible Cahn-Hilliard fluids and topological transitions. 
{\it R.\ Soc.\ Lond.\ Proc.\ Ser.\ A Math.\ Phys.\ Eng.\ Sci.}\ {\bf454} (1998), 2617--2654.

\bibitem[L2]{L2}
A.~Lunardi, 
Interpolation theory. Third edition.  
Appunti.\ Sc.\ Norm.\ Super.\ Pisa (N.~S.) {\bf16} [Lecture Notes. Scuola Normale Superiore di Pisa (New Series)].
{\it Edizioni della Normale, Pisa,} 2018.

\bibitem[MP]{MP}
J.~M\'alek and V.~Pr\r{u}\v{s}a, 
Derivation of equations for continuum mechanics and thermodynamics of fluids. 
{\it Handbook of mathematical analysis in mechanics of viscous fluids,} 3--72.
{\it Springer, Cham,} 2018.

\bibitem[MaS]{MaS}
A.~Marveggio and G.~Schimperna, 
On a non-isothermal Cahn-Hilliard model based on a microforce balance. 
{\it J.\ Differential Equations} {\bf274} (2021), 924--970.

\bibitem[Mi]{Mi}
A.~Miranville, 
The Cahn-Hilliard equation: Recent advances and applications.
CBMS-NSF Regional Conf.\ Ser.\ in Appl.\ Math., 95. 
{\it Society for Industrial and Applied Mathematics (SIAM), Philadelphia, PA,} 2019.

\bibitem[MS]{MS}
A.~Miranville and G.~Schimperna, 
Nonisothermal phase separation based on a microforce balance. 
{\it Discrete Contin.\ Dyn.\ Syst.\ Ser.\ B} {\bf5} (2005), no.~3, 753--768.

\bibitem[PS]{PS}
J.~Pr\"uss and G.~Simonett, 
Moving interfaces and quasilinear parabolic evolution equations. 
Monographs in Mathematics 105, 
{\it Birkh\"auser,} 2016.

\bibitem[PW]{PW}
M.~H.~Protter and H.~F.~Weinberger, 
Maximum principles in differential equations. 
Corrected reprint of the 1967 original. 
{\it Springer-Verlag, New York,} 1984.

\bibitem[Sa]{Sa}
Y.~Sawano, 
Theory of Besov spaces. 
Dev.\ Math.\ {\bf56}. 
{\it Springer, Singapore,} 2018.

\end{thebibliography}
\end{document}